\documentclass[11pt]{amsart}
\usepackage[utf8]{inputenc}
\usepackage{setspace}
\usepackage{geometry}
\usepackage{enumerate}
\usepackage{enumitem, xcolor, amssymb,latexsym,amsmath,bbm}
\usepackage{mathtools}
\usepackage[mathscr]{euscript}
\usepackage{amsmath}
\usepackage{amssymb}
\usepackage{enumitem}
\usepackage{xcolor}

\newcommand{\N}{\mathbb{N}}

\usepackage{tikz}
\usepackage{wrapfig}
\usepackage{enumerate}
\usepackage{graphicx}
\usepackage{subfigure}
\usepackage{float}

\usetikzlibrary{arrows.meta}

\usepackage[colorlinks=true,citecolor=blue,
linkcolor=blue,urlcolor=blue]{hyperref}
\definecolor{Red}{rgb}{1,0,0}

\calclayout

\theoremstyle{plain}

\newtheorem{theorem}{Theorem}[section]
\newtheorem*{repthm}{Theorem}
\newtheorem*{repprop}{Proposition}

\newtheorem{proposition}[theorem]{Proposition}
\theoremstyle{definition}
\newtheorem{definition}[theorem]{Definition}
\theoremstyle{remark}
\newtheorem{remark}[theorem]{Remark}
\newtheorem{note}[theorem]{Note}

\numberwithin{equation}{section}
\numberwithin{figure}{section}
\allowdisplaybreaks

\title{ Systolic embedding of graphs on translation  surfaces}
\author{Achintya Dey}

\address{
Department of Mathematics and Statistics\\ 
Indian Institute of Technology  \\ 
Kanpur, Uttar Pradesh-208016\\
India}
\email{achintd@iitk.ac.in}
\author{Bidyut Sanki}
\address{
Department of Mathematics and Statistics\\ 
Indian Institute of Technology  \\ 
Kanpur, Uttar Pradesh-208016\\
India}
\email{bidyut@iitk.ac.in}

\date{\today}
\begin{document}
\begin{abstract}
An embedding of a graph on a translation surface is said to be \emph{systolic} if each vertex of the graph corresponds to a singular point (or marked point) and each edge corresponds to a shortest saddle connection on the translation surface. The embedding is said to be \emph{cellular} (respectively \emph{essential}) if each complementary region is a topological disk (respectively not a topological disk). In this article, we prove that any finite graph admits an essential-systolic embedding on a translation surface and estimate the genera of such surfaces. For a wedge $\Sigma_n$ of $n$ circles, $n\geq2$, we investigate that $\Sigma_n$ admits cellular-systolic embedding on a translation surface and compute the minimum and maximum genera of such surfaces. Finally, we have identified another rich collection of graphs with more than one vertex that also admit cellular-sytolic embedding on translation surfaces.
\end{abstract}
\maketitle {}
\section{Introduction}
Embedding of graphs on surfaces has emerged as an active area of research in recent times. An embedding of a graph $G$ on a closed surface $S$ is called cellular (respectively, essential) if every component of the complement $S\setminus G$ is an open topological disk (respectively, not a topological disk). Embedding of graphs on closed hyperbolic surfaces has been studied by Sanki in~\cite{sa19} and Gadgil-Sanki in~\cite{san19}. In \cite{sa19}, Sanki has shown that any metric graph (up to scaling its metric) admits an essential and isometric embedding on a closed hyperbolic surface. Furthermore, he has determined the minimum genus of surfaces for such embeddings. In this paper, we study the systolic embedding of graphs on translation surfaces.

  A \emph{translation surface} is obtained from a finite union of Euclidean polygons by identifying parallel sides of opposite orientation via translations (see Definition \ref{Translation Surface}). The vertices of the polygons, after identification, give rise to cone points or singular points on a translation surface. Since the sides of the polygons are glued by Euclidean translations, the total angle at a point corresponding to the equivalence class of a vertex of a polygon is $2k\pi$, for some $k\in\mathbb N$. If $k>1$, then we call it a \emph{singular point}, and if $k=1$, the point is called a \emph{marked point}. Now, a \emph{saddle connection} on a translation surface is a geodesic segment between two singular points (which may not be distinct) such that it does not contain singular points in its interior. A \emph{shortest saddle connection} is a saddle connection of the shortest length. Note that for a translation surface without singular points, we use marked point in place of singular point to define these notions.
  
Let $S$ be a translation surface and $G$ be a graph on $S$. We say $G$ admits a \emph{systolic embedding} on $S$ if the vertices of $G$ are the singular points (or marked points) and the edges of $G$ are the shortest saddle connections on $S$.
Judge-Parlier~\cite{Pa19} have computed the maximum number of noncontractible loops of minimum length (systole) on a genus two translation surface. 
Boissy-Geninska~\cite{Bo21} have shown that the length of the shortest saddle connection of a translation surface of genus $g$ and having $r$ number of singular points (or marked points) is at most $\left[\frac{\sqrt{3}}{2}\left(2g-2+r\right)\right]^{-\frac{1}{2}}$ and the maximum number of shortest saddle connections is $3(2g-2+r)$. Later, Columbus-Herrlich-Muetzel-Schmithüsen~\cite{MR4790974} have studied the saddle connections graph of a translation surface and showed how the systoles of a translation surface and the systoles of its saddle connections graph are related.
 
Our article is mainly oriented towards the following problems:

\begin{enumerate}
      \item Given a finite graph $G$, does there exist a translation surface $S$ such that $G$ admits a systolic embedding on $S$?
    \item If such an embedding exists, then what is the minimum (and maximum) possible genus of a surface $S$ on which $G$ admits systolic embedding?
\end{enumerate}

We study two types of systolic embedding of graphs on translation surfaces: (i) essential and (ii) cellular. In the context of essential-systolic embedding, we establish that any finite graph $G$ admits essential-systolic embedding on some translation surface. In particular, we prove the theorem in the following.

\begin{theorem}\label{main theorem 1}
For every odd integer $n\geq 5$, the complete graph $K_n$ admits an essential-systolic embedding on a translation surface $S_{g_n}$ of the genus $g_n$, where $$g_n = 1 - n \left( 1 - \frac{n-1}{2} \right).$$
\end{theorem}

This result is established by explicitly building a translation surface \( S_{g_n} \) that admits such an embedding of \( K_n \). Furthermore, we establish the following proposition.

\begin{proposition}\label{proposition 1}
For odd $n$ with $n\geq 5$, every subgraph of $K_n$ admits essential-systolic embedding on a translation surface of genus $g_n$.

\end{proposition}

Since any finite graph \( G \) is a subgraph of a complete graph $K_n$ for some odd $n\ge5$, Theorem~\ref{main theorem 1} and Proposition~\ref{proposition 1} together yield the following main result.

\begin{theorem} \label{Main Result}
Every finite graph admits an essential-systolic embedding on a translation surface. 
\end{theorem}

If a graph admits an essential-systolic embedding on a translation surface of genus $g$, then there exists an integer $g'>g$ such that the graph admits such an embedding on a translation surface of genus $g'$. Therefore, the maximum-genus problem of such surfaces does not make sense, and hence we study the minimum-genus problem. We define $g^{\min}_G$ as the minimum possible genus of the surfaces on which the graph $G$ admits an essential-systolic embedding. To address Question~(2), we have established the following result.

\begin{theorem} \label{upper bound}
For a graph \( G \) with \( n \) vertices, an upper bound on the minimum genus of a translation surface that admits such an embedding is given by
\[
g^{\min}_G \leq 
\begin{cases}
1 - n\left(1 - \frac{n - 1}{2}\right), & \text{if } n \text{ is odd and } n \geq 5, \\[0.5em]
1 - (n + 1)\left(1 - \frac{n}{2}\right), & \text{if } n \text{ is even and } n \geq 5, \\[0.5em]
6, & \text{if } n < 5.
\end{cases}
\]
\end{theorem}

In contrast to essential embedding, not all graphs admit cellular-systolic embedding on translation surfaces. For example, the complete graph $K_3$ admits cellular embedding only on the sphere, which is not a translation surface. Thus, in response to Question (1), it is natural to ask whether there exists a collection of graphs that can be embedded in this way. In this context, we have identified a family of graphs $\Sigma_n$, which denotes the wedge of $n$ circles for $n \geq 2$. We establish the following results, which jointly answer questions (1) and (2) in this setting.

\begin{theorem} \label{cellular embedding}
Let $\Sigma_n$ be the wedge of $n$ circles, where $n \geq 2$. Then we have the following:
\begin{enumerate}
    \item $\Sigma_n$ admits cellular-systolic embedding on a translation surface of genus $\left\lfloor \frac{n}{2} \right\rfloor$. Furthermore, this is the maximum genus for which such an embedding is possible.
    
    \item The minimum genus of a translation surface on which $\Sigma_n$ admits cellular-systolic embedding is $\left\lceil \frac{n+3}{6} \right\rceil$.
    
    \item For any integer $g$ satisfying $\left\lceil \frac{n+3}{6} \right\rceil \leq g \leq \left\lfloor \frac{n}{2} \right\rfloor$, the graph $\Sigma_n$ admits  cellular-systolic embedding on a translation surface of the genus $g$.
\end{enumerate}
\end{theorem}

In addition, we identify another collection of graphs, each with more than one vertex, that can also be embedded on a translation surface in this manner.


\section{Graph Embeddings on Translation Surfaces}
\begin{definition}\label{Translation Surface} (Masur~\cite{Masur06})
A \emph{translation surface} is a finite union of Euclidean polygons $\{ \triangle_1, \triangle_2, \dots,\triangle_n \}$, equipped with a side pairing,
 such that:
\begin{enumerate}
    \item The boundary of every polygon is oriented so that the interior of the polygon lies to the left.
    \item For every $1\leq j\leq n$ and every oriented side $s_j$ of $\triangle_j$, there exists $k\in\{1,2,\cdots,n\}$ and an oriented side $s_k$ of $\triangle_k$ such that $s_j$ and $s_k$ are parallel and of the same length. They are glued together in the opposite orientation by a parallel translation.
\end{enumerate}
\end{definition}

By the definition of translation surfaces, the vertices of the Euclidean polygons used in their construction give rise to cone points or singular points. Since the sides of the polygons are identified by translations, the total angle around each cone point is a multiple of $2\pi$.

\begin{definition}
Let $X$ be a translation surface. A point on $X$ which represents the equivalence class of a vertex of a polygon is classified as follows:
\begin{enumerate}
    \item If the total angle around that point is greater than $2\pi$, it is called a \emph{singular point}.
    \item If the total angle is exactly $2\pi$, the point is called a \emph{marked point}.
\end{enumerate}
\end{definition}

For example, in Figure~\ref{Figure Torus}, the translation surface on the left (the square torus) has a marked with angle $2\pi$. The surface on the right is obtained by identifying parallel sides (with opposite orientation) of three unit squares. All vertices are identified to a single point with total angle $6\pi$, resulting in a singular point.

\begin{figure}[htbp]
\begin{center}
\begin{tikzpicture}[scale=1.5]
\draw (0,0)--(1,0)--(1,1)--(0,1)--(0,0);
\draw[red] (0,0) to (1,0);
\draw[red] ((1,1) to (0,1);
\draw[green] (0,0) to (0,1) ;
\draw[green] (1,0) to (1,1);
\draw (0,0) node{\tiny$\bullet$};
\draw (1,0) node{\tiny$\bullet$};
\draw (0,1) node{\tiny$\bullet$};
\draw (1,1) node{\tiny$\bullet$};
\draw [-{Stealth[red]}]
(0.55,0)--(0.56,0);
\draw [-{Stealth[red]}]
(0.55,1)--(0.56,1);
\draw [-{Stealth[green]}]
(0,0.55)--(0,0.56);
\draw [-{Stealth[green]}]
(1,0.55)--(1,0.56);
\draw (0.5,-0.2) node{$a_1$};
\draw (0.5,1.2) node {$a_1$};
\draw (-0.2, 0.5) node{$b_1$};
\draw (1.2, 0.5) node{$b_1$};

\draw (3,0)--(5,0)--(5,1)--(4,1)--(4,2)--(3,2)--(3,0);
\draw[red] (4,2) to (3,2);
\draw[red] (3,0) to (4,0);
\draw[green] (3,2) to (3,1);
\draw[green] (4,2) to (4,1);
\draw[blue] (3,1) to (3,0);
\draw[blue] (5,0) to (5,1);
\draw[dashed] (3,1) to (4,1);
\draw[dashed] (4,1) to (4,0);
\draw (3,0) node{\tiny$\bullet$};
\draw (4,0) node{\tiny$\bullet$};
\draw (5,0) node{\tiny$\bullet$};
\draw (5,1) node{\tiny$\bullet$};
\draw (4,1) node{\tiny$\bullet$};
\draw (3,1) node{\tiny$\bullet$};
\draw (3,2) node{\tiny$\bullet$};
\draw (4,2) node{\tiny$\bullet$};
\draw [-{Stealth[red]}]
(3.5,0)--(3.51,0);
\draw [-{Stealth[red]}]
(3.5,2)--(3.51,2);
\draw (3.5,-0.15) node{a};
\draw (3.5,2.15) node{a};
\draw [-{Stealth[blue]}]
(3,0.5)--(3,0.51);
\draw [-{Stealth[blue]}]
(3,0.5)--(3,0.51);
\draw [-{Stealth[blue]}]
(5,0.5)--(5,0.51);
\draw (2.85,0.5) node{b};
\draw (5.15,0.5) node{b};
\draw [-{Stealth[green]}]
(3,1.5)--(3,1.51);
\draw [-{Stealth[green]}]
(4,1.5)--(4,1.51);
\draw (2.85,1.5) node{c};
\draw (4.15,1.51) node{c};
\draw [-{Stealth}]
(4.5,0)--(4.51,0);
\draw [-{Stealth}]
(4.5,1)--(4.51,1);
\draw (4.5,-0.15) node{d};
\draw (4.5,1.15) node{d};

\end{tikzpicture}
\end{center}
\caption{Translation Surfaces with Marked and Singular Points }
\label{Figure Torus}
\end{figure}
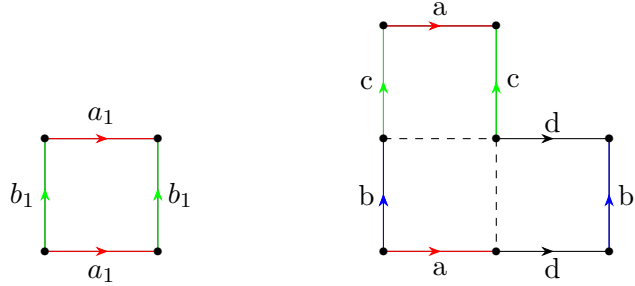

\begin{definition}
A \emph{saddle connection} on a translation surface is a geodesic segment connecting two singular points (which may coincide) with no singular points in its interior.
\end{definition}

A shortest saddle connection is a minimal-length saddle connection among all saddle connections on the surface. In this paper, we refer shortest saddle connection as \emph{systolic connection}.
\begin{note}
   When a translation surface has no singular points, a marked point replaces a singular point in these definitions.
\end{note}
In Figure~\ref{Figure Torus}, the segments $a_1$, $b_1$, and $a$, $b$, $c$, $d,$ are all saddle connections and are also examples of the systolic connections. For further details, we refer to~\cite{Mas22},~\cite{AR15} and~\cite{Masur06}.

\begin{definition}
The \emph{systolic graph} $\Gamma_X$, corresponding to a translation surface $X$, is a graph whose vertices are singular points (or marked points), and the edges are the systolic connections of $X$. 
\end{definition}

\begin{definition}
Let $X$ be a translation surface and $G$ be a graph.
\begin{enumerate}
\item  $G$ admits a \emph{cellular embedding} on $X$ if every component of $X\setminus G$ is an open topological disk.
\item G admits an \emph{essential embedding} on $X$ if none of the components of $X\setminus G$ is a topological disk.
\item  $G$ admits a \emph{systolic embedding} on $X$ if $G\cong \Gamma_X$.
\end{enumerate}
\end{definition}
\section{Essential-Systolic Embedding}
In this section, our goal is to prove Theorem~\ref{Main Result}. In particular, we prove that every graph admits an essential-systolic embedding on a translation surface. We begin with the proofs of Theorem~\ref{main theorem 1} and Proposition~\ref{proposition 1}, which are essential for the proof of Theorem~\ref{Main Result}.

\begin{repthm}[\ref{main theorem 1}]
    Every complete graph $K_n$, with $n$ odd and $n \geq 5$, admits an essential-systolic embedding on a translation surface $S_{g_n}$, where the genus $g_n$ is given by $$ g_n = 1 - n \left( 1 - \frac{n-1}{2} \right).$$
\end{repthm}

\begin{proof}
We prove the theorem by explicitly constructing a translation surface of the given genus on which the graph $K_n$ admits an essential-systolic embedding. We denote the vertices of $K_n$ by $v_0, v_1, \dots, v_{n-1}$.

To construct the desired translation surface, we begin by taking $n$ many regular Euclidean $(n-1)$-gons $R_0$, $R_1$, $\dots$, $R_{n-1}$, each with fixed area $a$ and side length $s$. Next, we introduce $n(n-1)$ rectangles $Q_{ij}$ (see Figure~\ref{Qij}), with sides labeled $a_{ij}, b_{ij}, a_{ij}^\prime, b_{ij}^\prime$, satisfying the condition:
\begin{equation}\label{eq:1}
    \ell(a_{ij}) = \ell(a_{ij}^\prime) = s \quad \text{and} \quad \ell(b_{ij}) = \ell(b_{ij}^\prime) < \frac{s}{2}, \quad \text{for } 0 \leq i \leq n-1,\ 1 \leq j \leq n-1.
\end{equation}

\begin{figure}[htbp]
\begin{center}
\begin{tikzpicture}[]
\draw (0,0)--(2,0);
\draw[red] (2,0)--(2,1);
\draw (2,1)--(0,1);
\draw[red] (0,1)--(0,0);

\draw (1, 1.2) node{$a_{ij}$};
\draw (1,-0.3) node {$a^\prime_{ij}$};
\draw (-0.3,0.5) node{$b_{ij}$};
\draw (2.3,0.5) node {$b^\prime_{ij}$};
\end{tikzpicture}
\end{center}
\caption{$Q_{ij}$ }
\label{Qij}
\end{figure}

We attach each rectangle $Q_{ij}$ to the polygon $R_i$ along the side $a_{ij}^\prime$ by a translation. Let $P_0, P_1, \dots, P_{n-1}$ be the resulting figures composed of $R_i$ and the corresponding rectangles. In each $P_i$, we identify the sides $b_{ij}$ with $b_{ij}^\prime$ by translation.

To complete the construction of the translation surface, we now glue the remaining sides $a_{ij}$ as described below:

We obtain ${\frac{n-1}{2}}$ edge-disjoint Hamiltonian cycles $C_l$ in $K_n$ using Walecki's construction (see~\cite{MR4089838} Lemma 18):
$$C_l= \left(v_0, v_l, v_{l+1}, v_{l-1}, v_{l+2},\ v_{l-2},\dots,  v_{l+\frac{n-1}{2}}\right), \text{ for }1 \leq l \leq \frac{n-1}{2},$$ where subscripts of $v$'s are taken modulo $n-1$. Now, consider the following permutations in $S_n$ corresponding to the above cycles: $$\sigma_l^\prime = \left(0,l, l+1, l-1, l+2, l-2, \dots, l+\frac{n-1}{2} \right) ~mod~(n-1) \text{ for }1 \leq l \leq \frac{n-1}{2}.$$
 For each  $P_i$, re-label the sides $a_{ij}$ as $E_1(i), E_2(i), \dots, E_{\frac{n-1}{2}}(i),\bar{E}_1(i), \bar{E}_2(i), \dots, \bar{E}_{\frac{n-1}{2}}(i)$ in a fixed order such that each \( E_r(i) \) is parallel to \( \bar{E}_r(i) \). Moreover, for any \( 0 \leq i, i' \leq n-1 \) and \( 1 \leq r \leq \frac{n-1}{2} \), we have $E_r(i)$ is parallel to $\bar{E}_r(i')$. We glue the sides $E_r(i)$ with $\bar{E}_r(\sigma_r^\prime(i))$, for $i = 0, 1, \dots, n-1$ and $r = 1, 2, \dots, \frac{n-1}{2}$. This process yields a translation surface $S_{g_n}$, with $n$ singular points, each having total angle $2\pi(n - 2)$. If $\sigma_r^\prime(i) = i'$, then the side $E_r(i) = a_{ij}$ is glued to $\bar{E}_r(i') = a_{i'j'}$ for some $j, j' \in \{1, 2, \dots, n-1\}$. Thus, the concatenation $b_{ij} * b_{i'j'}^\prime$ forms a saddle connection of $S_{g_n}$. From condition~\eqref{eq:1}, the length of each such saddle connection is less than $s$, making it the systolic connection. For instance, we explicitly construct a translation surface for \( K_5 \) (Figure~\ref{figure18}) by the method described above, using the permutations \( \sigma'_1 = (0\ 1\ 2\ 4\ 3) \) and \( \sigma'_2 = (0\ 2\ 3\ 1\ 4) \). The resulting surface is illustrated in Figure~\ref{figure19}.

Therefore, all these $n(n-1)$ saddle connections are systolic connections, and each singular point connects to $(n-1)$ others. The systolic graph $\Gamma_{S_{g_n}}$ is a $(n-1)$-regular graph with $n$ vertices $V_0, V_1,\dots, V_{n-1}$. Now, we map $V_i$ to $v_i$ for each $0\le i\le n-1$ and hence the above construction ensures that $\Gamma_{S_{g_n}}\cong K_n$. Since $\Gamma_{S_{g_n}} \cong K_n$, and the components of $S_{g_n} \setminus K_n$ are not open topological disks, it follows that $K_n$ admits an essential-systolic embedding on $S_{g_n}$.

Finally, using the Euler characteristic formula, we obtain: $$g_n = \left[1 - \left(1 - \frac{n-1}{2} \right)n\right].$$

\end{proof}

\begin{figure}[htbp]
\begin{center}
\begin{tikzpicture}[scale=2]
\draw (-0.587,-0.809)--(0.587,-0.809);
\draw(0.587,-0.809)--(0.951,0.309);
\draw (0.951,0.309)--(0,1);
\draw (0,1)--(-0.951,0.309);
\draw (-0.951,0.309)--(-0.587,-0.809);

\draw[red] (-0.587,-0.809) node{\tiny$\bullet$};
\draw (-0.75,-0.809) node{$v_0$};
\draw[blue] (0.587,-0.809) node{\tiny$\bullet$};
\draw (0.75,-0.809) node{$v_1$};
\draw[green] (0.951,0.309) node{\tiny$\bullet$};
\draw[] (1.1,0.309) node{$v_2$};
\draw[yellow] (0,1) node{\tiny$\bullet$};
\draw[] (0,1.1) node{$v_3$};
\draw (-0.951,0.309) node{\tiny$\bullet$};
\draw (-1.1,0.309) node{$v_4$};

\draw (-0.587,-0.809)--(0.951,0.309);
\draw (-0.587,-0.809)--(0,1);
\draw (0.951,0.309)--(-0.951,0.309)--(0.587,-0.809)--(0,1);

\end{tikzpicture}
\caption{$K_5$}
\label{figure18}
\end{center}
\end{figure}

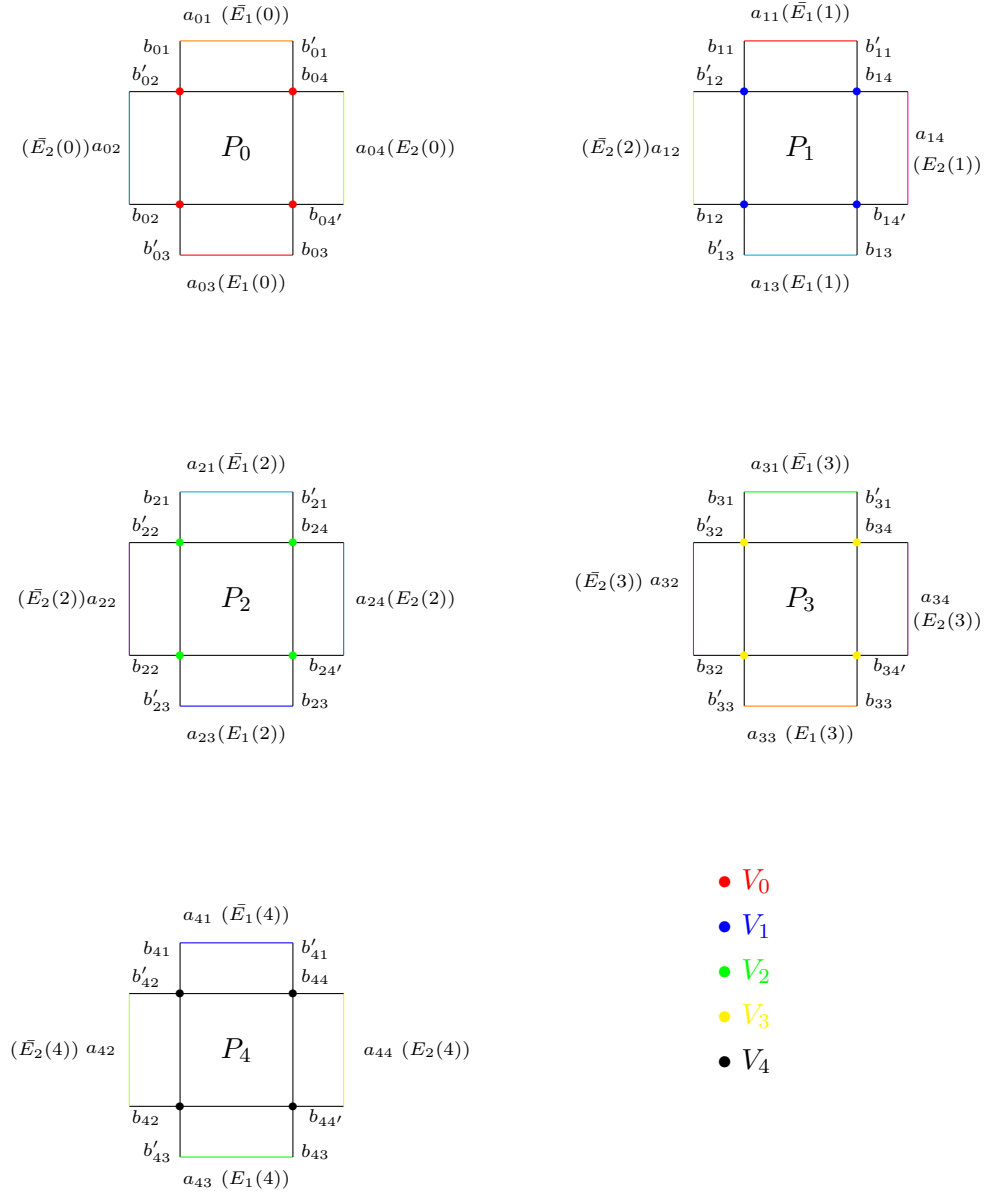
\begin{figure}[htbp]
\begin{center}
\begin{tikzpicture}[scale=1.5]
\draw (0,0)--(1,0)--(1,1)--(0,1)--(0,0);
\draw (0,0)--(0,-0.45);
\draw[red] (0,-0.45)--(1,-0.45);
\draw (1,-0.45)--(1,0);
\draw (1,0)--(1.45,0);
\draw[lime] (1.45,0)--(1.45,1);
\draw (1.45,1)--(1,1);
\draw (0,1)--(0,1.45);
\draw[orange] (0,1.45)--(1,1.45);
\draw (1,1.45)--(1,1);
\draw (0,0)--(-0.45,0);
\draw[teal] (-0.45,0)--(-0.45,1);
\draw (-0.45,1)--(0,1);

\draw (-0.2,1.4) node{\tiny$b_{01}$};
\draw (1.2,1.4) node {\tiny$b^\prime_{01}$};
\draw (-0.3,-0.1) node{\tiny$b_{02}$};
\draw (-0.3,1.15) node{\tiny$b_{02}^\prime$};
\draw (1.2,1.15) node{\tiny$b_{04}$};
\draw (1.3,-0.1) node{\tiny$b_{04^\prime}$};
\draw (-0.2,-.4) node{\tiny$b_{03}^\prime$};
\draw (1.2,-.4) node{\tiny$b_{03}$};

\draw (0.5,1.68) node{\tiny$a_{01}$ $(\bar{E_1}(0))$};
\draw (-0.65,0.5) node{\tiny$a_{02}$};
\draw (-1.1,0.5) node {\tiny$(\bar{E_2}(0))$};
\draw (0.5,-0.7) node{\tiny$a_{03}  (E_1(0))$};
\draw (2,0.5) node{\tiny$a_{04} (E_2(0))$};
\draw (0.5,0.5)node{$P_0$};
\draw[red] (0,0) node{\tiny$\bullet$};
\draw[red] (1,0) node{\tiny$\bullet$};
\draw[red] (1,1) node{\tiny$\bullet$};
\draw[red] (0,1) node{\tiny$\bullet$};

\draw (5,0)--(6,0)--(6,1)--(5,1)--(5,0);
\draw (5,0)--(5,-0.45);
\draw[cyan] (5,-0.45)--(6,-0.45);
\draw (6,-0.45)--(6,0);
\draw (6,0)--(6.45,0);
\draw[magenta] (6.45,0)--(6.45,1);
\draw (6.45,1)--(6,1);
\draw (5,1)--(5,1.45);
\draw[red] (5,1.45)--(6,1.45);
\draw (6,1.45)--(6,1);
\draw (4.8,1.4) node{\tiny$b_{11}$};
\draw (6.2,1.4) node {\tiny$b^\prime_{11}$};
\draw (5,0)--(4.55,0);
\draw[yellow] (4.55,0)--(4.55,1);
\draw(4.55,1)--(5,1);
\draw (4.7,-0.1) node{\tiny$b_{12}$};
\draw (4.7,1.15) node{\tiny$b_{12}^\prime$};
\draw (6.2,1.15) node{\tiny$b_{14}$};
\draw (6.3,-0.1) node{\tiny$b_{14^\prime}$};
\draw (4.8,-.4) node{\tiny$b_{13}^\prime$};
\draw (6.2,-.4) node{\tiny$b_{13}$};

\draw (5.5,1.7) node{\tiny$a_{11} (\bar{E_1}(1))$};
\draw (4,0.5) node{\tiny$(\bar{E_2}(2)) a_{12}$};

\draw (5.5,-0.7) node{\tiny$a_{13} (E_1(1))$};
\draw (6.65,0.6) node{\tiny$a_{14}$};
\draw (6.8,0.35) node{\tiny$ (E_2(1))$};
\draw (5.5,0.5)node{$P_1$};
\draw[blue] (5,0) node{\tiny$\bullet$};
\draw[blue] (6,0) node{\tiny$\bullet$};
\draw[blue] (6,1) node{\tiny$\bullet$};
\draw[blue] (5,1) node{\tiny$\bullet$};

\draw (0,-4)--(1,-4)--(1,-3)--(0,-3)--(0,-4);
\draw (0,-4)--(0,-4.45);
\draw[blue] (0,-4.45)--(1,-4.45);
\draw(1,-4.45)--(1,-4);
\draw (1,-4)--(1.45,-4);
\draw[teal] (1.45,-4)--(1.45,-3);
\draw (1.45,-3)--(1,-3);
\draw (0,-3)--(0,-2.55);
\draw[cyan] (0,-2.55)--(1,-2.55);
\draw(1,-2.55)--(1,-3);
\draw (0,-4)--(-0.45,-4);
\draw[violet] (-0.45,-4)--(-0.45,-3);
\draw[] (-0.45,-3)--(0,-3);
\draw (-0.2,-2.6) node{\tiny$b_{21}$};
\draw (1.2,-2.6) node {\tiny$b^\prime_{21}$};
\draw (-0.3,-4.1) node{\tiny$b_{22}$};
\draw (-0.3,-2.85) node{\tiny$b_{22}^\prime$};
\draw (1.2,-2.85) node{\tiny$b_{24}$};
\draw (1.3,-4.1) node{\tiny$b_{24^\prime}$};
\draw (-0.2,-4.4) node{\tiny$b_{23}^\prime$};
\draw (1.2,-4.4) node{\tiny$b_{23}$};

\draw (0.5,-2.3) node{\tiny$a_{21} (\bar{E_1}(2))$};
\draw (-1,-3.5) node{\tiny$(\bar{E_2}(2)) a_{22} $};
\draw (0.5,-4.7) node{\tiny$a_{23} (E_1(2))$};
\draw (2,-3.5) node{\tiny$a_{24}(E_2(2))$};
\draw (0.5,-3.5)node{$P_2$};
\draw[green] (0,-4) node{\tiny$\bullet$};
\draw[green] (1,-4) node{\tiny$\bullet$};
\draw[green] (1,-3) node{\tiny$\bullet$};
\draw[green] (0,-3) node{\tiny$\bullet$};

\draw (5,-4)--(6,-4)--(6,-3)--(5,-3)--(5,-4);
\draw (5,-4)--(5,-4.45);
\draw[orange] (5,-4.45)--(6,-4.45);
\draw(6,-4.45)--(6,-4);
\draw (6,-4)--(6.45,-4);
\draw[violet] (6.45,-4)--(6.45,-3);
\draw (6.45,-3)--(6,-3);
\draw (5,-3)--(5,-2.55);
\draw[green] (5,-2.55)--(6,-2.55);
\draw(6,-2.55)--(6,-3);
\draw (5,-4)--(4.55,-4);
\draw[magenta] (4.55,-4)--(4.55,-3);
\draw(4.55,-3)--(5,-3);
\draw (4.8,-2.6) node{\tiny$b_{31}$};
\draw (6.2,-2.6) node {\tiny$b^\prime_{31}$};
\draw (4.7,-4.1) node{\tiny$b_{32}$};
\draw (4.7,-2.85) node{\tiny$b_{32}^\prime$};
\draw (6.2,-2.85) node{\tiny$b_{34}$};
\draw (6.3,-4.1) node{\tiny$b_{34^\prime}$};
\draw (4.8,-4.4) node{\tiny$b_{33}^\prime$};
\draw (6.2,-4.4) node{\tiny$b_{33}$};

\draw (5.5,-2.3) node{\tiny$a_{31} (\bar{E_1}(3))$};
\draw (4.3,-3.35) node{\tiny$a_{32}$};
\draw (3.8,-3.35) node{\tiny$(\bar{E_2}(3))$};
\draw (5.5,-4.7) node{\tiny$a_{33}$ $ (E_1(3))$};
\draw (6.7,-3.5) node{\tiny$a_{34}$};
\draw (6.8,-3.7) node{\tiny$(E_2(3))$};

\draw (5.5,-3.5)node{$P_3$};
\draw[yellow] (5,-4) node{\tiny$\bullet$};
\draw[yellow] (6,-4) node{\tiny$\bullet$};
\draw[yellow] (6,-3) node{\tiny$\bullet$};
\draw[yellow] (5,-3) node{\tiny$\bullet$};

\draw (0,-8)--(1,-8)--(1,-7)--(0,-7)--(0,-8);
\draw (0,-8)--(0,-8.45);
\draw[green] (0,-8.45)--(1,-8.45);
\draw (1,-8.45)--(1,-8);
\draw (1,-8)--(1.45,-8);
\draw[yellow] (1.45,-8)--(1.45,-7);
\draw (1.45,-7)--(1,-7);
\draw (0,-7)--(0,-6.55);
\draw[blue] (0,-6.55)--(1,-6.55);
\draw (1,-6.55)--(1,-7);
\draw (0,-8)--(-0.45,-8);
\draw[lime] (-0.45,-8)--(-0.45,-7);
\draw (-0.45,-7)--(0,-7);
\draw (-0.2,-6.6) node{\tiny$b_{41}$};
\draw (1.2,-6.6) node {\tiny$b^\prime_{41}$};
\draw (-0.3,-8.1) node{\tiny$b_{42}$};
\draw (-0.3,-6.85) node{\tiny$b_{42}^\prime$};
\draw (1.2,-6.85) node{\tiny$b_{44}$};
\draw (1.3,-8.1) node{\tiny$b_{44^\prime}$};
\draw (-0.2,-8.4) node{\tiny$b_{43}^\prime$};
\draw (1.2,-8.4) node{\tiny$b_{43}$};

\draw (0.5,-6.3) node{\tiny$a_{41}$ $(\bar{E_1}(4))$};
\draw (-0.7,-7.5) node{\tiny$a_{42}$};
\draw (-1.2,-7.5) node {\tiny$(\bar{E_2}(4))$};
\draw (0.5,-8.65) node{\tiny$a_{43}$ $ (E_1(4))$};
\draw (2.1,-7.5) node{\tiny$a_{44}$ $(E_2(4))$};
\draw (0.5,-7.5)node{$P_4$};
\draw[] (0,-8) node{\tiny$\bullet$};
\draw[] (1,-8) node{\tiny$\bullet$};
\draw[] (1,-7) node{\tiny$\bullet$};
\draw[] (0,-7) node{\tiny$\bullet$};

\draw[red] (5,-6) node{$\bullet$ $V_0$};
\draw[blue] (5,-6.4) node{$\bullet$ $V_1$};
\draw[green] (5,-6.8) node{$\bullet$ $V_2$};
\draw[yellow] (5,-7.2) node{$\bullet$ $V_3$};
\draw (5,-7.6) node{$\bullet$ $V_4$};
\end{tikzpicture}
\caption{Translation surface for $K_5$}
\label{figure19}
\end{center}
\end{figure}

Now, we are ready to prove the proposition \ref{proposition 1}.

\begin{repprop}[\ref{proposition 1}]
  For odd $n$ with $n\geq 5$, every subgraph of $K_n$ admits essential-systolic embedding on a translation surface of genus $g_n$.
\end{repprop}
\begin{proof}
Let $e$ be an edge of $K_n$. We show that the graph $K_n \setminus \{e\}$ admits a essential-systolic embedding on a translation surface of the same genus $g_n$.

Consider the embedding of $K_n$ on $S_{g_n}$ as given in Theorem~\ref{main theorem 1}. Assume that the edge $e$ is formed by the concatenation $b_{ij} * b_{ml}^\prime$, for some fixed indices $i, j, m, l$, with $\ell(b_{ij}) = \ell(b_{ml}^\prime) = \ell_0.$ Now, we modify the rectangle $Q_{ij}$ in such a way that
$
\ell(b_{ij}) = \ell(b_{ij}^\prime) > \ell_0$. This ensures that the concatenation $b_{ij} * b_{ml}^\prime$ is no longer a systolic connection.

By performing this modification on $Q_{ij}$ and keeping the rest of the construction unchanged, we obtain a new translation surface $S_{g_n}^*$, of the genus $g_n$, on which $K_n \setminus \{e\}$ is essentially and systolically embedded.

Now it follows that any subgraph $H \subseteq K_n$ admits essential-systolic embedding on a translation surface of the genus $g_n$.
\end{proof}
In Figure~\ref{figure21}, we show the construction for $K_5\setminus\{e_{v_0v_1}\}$ (Figure~\ref{figure20}). 
\begin{figure}[htbp]
\begin{center}
\begin{tikzpicture}[scale=2]

\draw(0.587,-0.809)--(0.951,0.309);
\draw (0.951,0.309)--(0,1);
\draw (0,1)--(-0.951,0.309);
\draw (-0.951,0.309)--(-0.587,-0.809);

\draw[red] (-0.587,-0.809) node{\tiny$\bullet$};
\draw (-0.75,-0.809) node{$v_0$};
\draw[blue] (0.587,-0.809) node{\tiny$\bullet$};
\draw (0.75,-0.809) node{$v_1$};
\draw[green] (0.951,0.309) node{\tiny$\bullet$};
\draw[] (1.1,0.309) node{$v_2$};
\draw[yellow] (0,1) node{\tiny$\bullet$};
\draw[] (0,1.1) node{$v_3$};
\draw (-0.951,0.309) node{\tiny$\bullet$};
\draw (-1.1,0.309) node{$v_4$};

\draw (-0.587,-0.809)--(0.951,0.309);
\draw (-0.587,-0.809)--(0,1);
\draw (0.951,0.309)--(-0.951,0.309)--(0.587,-0.809)--(0,1);

\end{tikzpicture}
\caption{$K_5\setminus \{e_{v_0 v_1}\}$}
\label{figure20}
\end{center}
\end{figure}

\begin{figure}[htbp]
\begin{center}
\begin{tikzpicture}[scale=1.5]
\draw (0,0)--(1,0)--(1,1)--(0,1)--(0,0);
\draw (0,0)--(0,-0.8);
\draw[red] (0,-0.8)--(1,-0.8);
\draw (1,-0.8)--(1,0);

\draw (1,0)--(1.45,0);
\draw[lime] (1.45,0)--(1.45,1);
\draw (1.45,1)--(1,1);
\draw (0,1)--(0,1.45);
\draw[orange] (0,1.45)--(1,1.45);
\draw (1,1.45)--(1,1);
\draw (0,0)--(-0.45,0);
\draw[teal] (-0.45,0)--(-0.45,1);
\draw (-0.45,1)--(0,1);

\draw (-0.2,1.4) node{\tiny$b_{01}$};
\draw (1.2,1.4) node {\tiny$b^\prime_{01}$};
\draw (-0.3,-0.1) node{\tiny$b_{02}$};
\draw (-0.3,1.15) node{\tiny$b_{02}^\prime$};
\draw (1.2,1.15) node{\tiny$b_{04}$};
\draw (1.3,-0.1) node{\tiny$b_{04^\prime}$};
\draw (-0.2,-0.6) node{\tiny$b_{03}^\prime$};
\draw (1.2,-0.6) node{\tiny$b_{03}$};

\draw (0.5,1.68) node{\tiny$a_{01}$ $(\bar{E_1}(0))$};
\draw (-0.65,0.5) node{\tiny$a_{02}$};
\draw (-1.1,0.5) node {\tiny$(\bar{E_2}(0))$};
\draw (0.5,-1) node{\tiny$a_{03}  (E_1(0))$};
\draw (2,0.5) node{\tiny$a_{04} (E_2(0))$};
\draw (0.5,0.5)node{$P_0$};
\draw[red] (0,0) node{\tiny$\bullet$};
\draw[red] (1,0) node{\tiny$\bullet$};
\draw[red] (1,1) node{\tiny$\bullet$};
\draw[red] (0,1) node{\tiny$\bullet$};

\draw (5,0)--(6,0)--(6,1)--(5,1)--(5,0);
\draw (5,0)--(5,-0.45);
\draw[cyan] (5,-0.45)--(6,-0.45);
\draw (6,-0.45)--(6,0);
\draw (6,0)--(6.45,0);
\draw[magenta] (6.45,0)--(6.45,1);
\draw (6.45,1)--(6,1);
\draw (5,1)--(5,1.45);
\draw[red] (5,1.45)--(6,1.45);
\draw (6,1.45)--(6,1);
\draw (4.8,1.4) node{\tiny$b_{11}$};
\draw (6.2,1.4) node {\tiny$b^\prime_{11}$};
\draw (5,0)--(4.55,0);
\draw[yellow] (4.55,0)--(4.55,1);
\draw(4.55,1)--(5,1);
\draw (4.7,-0.1) node{\tiny$b_{12}$};
\draw (4.7,1.15) node{\tiny$b_{12}^\prime$};
\draw (6.2,1.15) node{\tiny$b_{14}$};
\draw (6.3,-0.1) node{\tiny$b_{14^\prime}$};
\draw (4.8,-.4) node{\tiny$b_{13}^\prime$};
\draw (6.2,-.4) node{\tiny$b_{13}$};

\draw (5.5,1.7) node{\tiny$a_{11} (\bar{E_1}(1))$};
\draw (4,0.5) node{\tiny$(\bar{E_2}(2)) a_{12}$};

\draw (5.5,-0.8) node{\tiny$a_{13} (E_1(1))$};
\draw (6.65,0.6) node{\tiny$a_{14}$};
\draw (6.8,0.35) node{\tiny$ (E_2(1))$};
\draw (5.5,0.5)node{$P_1$};
\draw[blue] (5,0) node{\tiny$\bullet$};
\draw[blue] (6,0) node{\tiny$\bullet$};
\draw[blue] (6,1) node{\tiny$\bullet$};
\draw[blue] (5,1) node{\tiny$\bullet$};

\draw (0,-4)--(1,-4)--(1,-3)--(0,-3)--(0,-4);
\draw (0,-4)--(0,-4.45);
\draw[blue] (0,-4.45)--(1,-4.45);
\draw(1,-4.45)--(1,-4);
\draw (1,-4)--(1.45,-4);
\draw[teal] (1.45,-4)--(1.45,-3);
\draw (1.45,-3)--(1,-3);
\draw (0,-3)--(0,-2.55);
\draw[cyan] (0,-2.55)--(1,-2.55);
\draw(1,-2.55)--(1,-3);
\draw (0,-4)--(-0.45,-4);
\draw[violet] (-0.45,-4)--(-0.45,-3);
\draw[] (-0.45,-3)--(0,-3);

\draw (-0.2,-2.6) node{\tiny$b_{21}$};
\draw (1.2,-2.6) node {\tiny$b^\prime_{21}$};
\draw (-0.3,-4.1) node{\tiny$b_{22}$};
\draw (-0.3,-2.85) node{\tiny$b_{22}^\prime$};
\draw (1.2,-2.85) node{\tiny$b_{24}$};
\draw (1.3,-4.1) node{\tiny$b_{24^\prime}$};
\draw (-0.2,-4.4) node{\tiny$b_{23}^\prime$};
\draw (1.2,-4.4) node{\tiny$b_{23}$};

\draw (0.5,-2.3) node{\tiny$a_{21} (\bar{E_1}(2))$};
\draw (-1,-3.5) node{\tiny$(\bar{E_2}(2)) a_{22} $};
\draw (0.5,-4.6) node{\tiny$a_{23} (E_1(2))$};
\draw (2,-3.5) node{\tiny$a_{24}(E_2(2))$};
\draw (0.5,-3.5)node{$P_2$};
\draw[green] (0,-4) node{\tiny$\bullet$};
\draw[green] (1,-4) node{\tiny$\bullet$};
\draw[green] (1,-3) node{\tiny$\bullet$};
\draw[green] (0,-3) node{\tiny$\bullet$};

\draw (5,-4)--(6,-4)--(6,-3)--(5,-3)--(5,-4);
\draw (5,-4)--(5,-4.45);
\draw[orange] (5,-4.45)--(6,-4.45);
\draw(6,-4.45)--(6,-4);
\draw (6,-4)--(6.45,-4);
\draw[violet] (6.45,-4)--(6.45,-3);
\draw (6.45,-3)--(6,-3);
\draw (5,-3)--(5,-2.55);
\draw[green] (5,-2.55)--(6,-2.55);
\draw(6,-2.55)--(6,-3);
\draw (5,-4)--(4.55,-4);
\draw[magenta] (4.55,-4)--(4.55,-3);
\draw(4.55,-3)--(5,-3);

\draw (4.8,-2.6) node{\tiny$b_{31}$};
\draw (6.2,-2.6) node {\tiny$b^\prime_{31}$};
\draw (4.7,-4.1) node{\tiny$b_{32}$};
\draw (4.7,-2.85) node{\tiny$b_{32}^\prime$};
\draw (6.2,-2.85) node{\tiny$b_{34}$};
\draw (6.3,-4.1) node{\tiny$b_{34^\prime}$};
\draw (4.8,-4.4) node{\tiny$b_{33}^\prime$};
\draw (6.2,-4.4) node{\tiny$b_{33}$};

\draw (5.5,-2.3) node{\tiny$a_{31} (\bar{E_1}(3))$};
\draw (4.3,-3.35) node{\tiny$a_{32}$};
\draw (3.8,-3.35) node{\tiny$(\bar{E_2}(3))$};
\draw (5.5,-4.6) node{\tiny$a_{33}$ $ (E_1(3))$};
\draw (6.7,-3.5) node{\tiny$a_{34}$};
\draw (6.8,-3.7) node{\tiny$(E_2(3))$};

\draw (5.5,-3.5)node{$P_3$};
\draw[yellow] (5,-4) node{\tiny$\bullet$};
\draw[yellow] (6,-4) node{\tiny$\bullet$};
\draw[yellow] (6,-3) node{\tiny$\bullet$};
\draw[yellow] (5,-3) node{\tiny$\bullet$};

\draw (0,-8)--(1,-8)--(1,-7)--(0,-7)--(0,-8);
\draw (0,-8)--(0,-8.45);
\draw[green] (0,-8.45)--(1,-8.45);
\draw (1,-8.45)--(1,-8);
\draw (1,-8)--(1.45,-8);
\draw[yellow] (1.45,-8)--(1.45,-7);
\draw (1.45,-7)--(1,-7);
\draw (0,-7)--(0,-6.55);
\draw[blue] (0,-6.55)--(1,-6.55);
\draw (1,-6.55)--(1,-7);
\draw (0,-8)--(-0.45,-8);
\draw[lime] (-0.45,-8)--(-0.45,-7);
\draw (-0.45,-7)--(0,-7);

\draw (-0.2,-6.6) node{\tiny$b_{41}$};
\draw (1.2,-6.6) node {\tiny$b^\prime_{41}$};
\draw (-0.3,-8.1) node{\tiny$b_{42}$};
\draw (-0.3,-6.85) node{\tiny$b_{42}^\prime$};
\draw (1.2,-6.85) node{\tiny$b_{44}$};
\draw (1.3,-8.1) node{\tiny$b_{44^\prime}$};
\draw (-0.2,-8.4) node{\tiny$b_{43}^\prime$};
\draw (1.2,-8.4) node{\tiny$b_{43}$};

\draw (0.5,-6.3) node{\tiny$a_{41}$ $(\bar{E_1}(4))$};
\draw (-0.7,-7.5) node{\tiny$a_{42}$};
\draw (-1.2,-7.5) node {\tiny$(\bar{E_2}(4))$};
\draw (0.5,-8.65) node{\tiny$a_{43}$ $ (E_1(4))$};
\draw (2.1,-7.5) node{\tiny$a_{44}$ $(E_2(4))$};
\draw (0.5,-7.5)node{$P_4$};
\draw[] (0,-8) node{\tiny$\bullet$};
\draw[] (1,-8) node{\tiny$\bullet$};
\draw[] (1,-7) node{\tiny$\bullet$};
\draw[] (0,-7) node{\tiny$\bullet$};

\draw[red] (5,-6) node{$\bullet$ $V_0$};
\draw[blue] (5,-6.4) node{$\bullet$ $V_1$};
\draw[green] (5,-6.8) node{$\bullet$ $V_2$};
\draw[yellow] (5,-7.2) node{$\bullet$ $V_3$};
\draw (5,-7.6) node{$\bullet$ $V_4$};
\end{tikzpicture}
\caption{Translation surface for $K_5\setminus\{e_{v_0v_1}\}$}
\label{figure21}
\end{center}
\end{figure}

\begin{repthm}[\ref{Main Result}]
  Every finite graph admits a essential-systolic embedding on a translation surface.  
\end{repthm}
\begin{proof}
Any graph $G$ with $m$ vertices is a subgraph of the complete graph $K_n$, for an odd integer $n$. By Proposition~\ref{proposition 1}, the theorem now follows.
\end{proof}

 We define $P_G$, the set of all genera $g$ such that $G$ admits an essential-systolic embedding on a translation surface of the genus $g$, which is nonempty by Theorem \ref{Main Result}. 
 \begin{remark}
If a graph $G$ admits an essential-systolic embedding on a translation surface of genus $g$, then there exists $g^\prime > g$ such that $G$ also admits such an embedding on a translation surface of genus $g^\prime$. Consequently, the notion of a maximum genus does not make sense, and we therefore focus on determining the minimum genus for such embeddings.

 \end{remark}
  Now, we define
$g^{\min}_G$ := $\min \left\{ g \mid g \in P_G \right\}$. The following result provides an upper bound for $g^{\min}_G$.

\begin{repthm}[\ref{upper bound}]
For a graph \( G \) with \( n \) vertices, an upper bound on the minimum genus of a translation surface that admits such an embedding is given by
 \[
g^{\min}_G \leq 
\begin{cases}
1 - n\left(1 - \frac{n - 1}{2}\right), & \text{if } n \geq 5 \text{ and } n \text{ is odd}, \\[0.5em]
1 - (n + 1)\left(1 - \frac{n}{2}\right), & \text{if } n \geq 5 \text{ and } n \text{ is even and} \\[0.5em]
6, & \text{if } n < 5.
\end{cases}
\]

\end{repthm}

\begin{proof}

For $n \geq 5$, we consider the graph $G$ as a subgraph of the complete graph $K_n$ when $n$ is odd, or $K_{n+1}$ when $n$ is even. Then, Theorem~\ref{main theorem 1} and Proposition~\ref{proposition 1} together imply that $G$ admits essential-systolic embedding on a translation surface of genus $g_n$ or $g_{n+1}$, depending on whether $n$ is odd or even, respectively. Therefore, we have: $$
g^{\min}_G \leq g_n \quad \text{if } n \text{ is odd,} \quad \text{and} \quad g^{\min}_G \leq g_{n+1} \quad \text{if } n \text{ is even}.$$
Now, by Theorem~\ref{main theorem 1}, $$g_n = 1 - n\left(1 - \frac{n - 1}{2}\right), \quad \text{and} \quad g_{n+1} = 1 - (n + 1)\left(1 - \frac{n}{2}\right).$$

Similarly, for $n < 5$, we take $G$ as a subgraph of $K_5$. Using the same arguments, we conclude that:$$g^{\min}_G \leq g_{5} = 6.$$

\end{proof}

\section{Cellular-systolic embedding}

In this section, we study the cellular-systolic embedding of graphs on translation surfaces. Consider the graph $\Sigma_n$, wedge of $n$-circles. Note that for $n=1$, the graph $\Sigma_1$ does not admit cellular systolic embedding on any translation surface. To see this, if $\Sigma_1$ is cellularly embedded on a closed surface, then that surface is topologically a sphere, which is not a translation surface. In the following theorem, we show that, for $n\geq 2$, $\Sigma_n$ admits cellular-systolic embedding on a translation surface.

\begin{repthm}[\ref{cellular embedding}]
Let \( \Sigma_n \) denote the wedge of \( n \) circles, where \( n \geq 2 \). Then we have the following:
\begin{enumerate}
    \item $\Sigma_n$ admits a cellular-systolic embedding on a translation surface of genus $\left\lfloor \frac{n}{2} \right\rfloor$. Furthermore, this is the maximum possible genus of such a translation surface.
    
    \item The minimum genus of a translation surface on which $\Sigma_n$ admits a cellular-systolic embedding is \( \left\lceil \frac{n+3}{6} \right\rceil \).
    
    \item For any integer \( g \) satisfying $\left\lceil \frac{n+3}{6} \right\rceil \leq g \leq \left\lfloor \frac{n}{2} \right\rfloor$, the graph \( \Sigma_n \) admits cellular-systolic embedding on a translation surface of genus \( g \).
\end{enumerate}
\end{repthm}
\begin{proof}
\begin{enumerate}
\item We prove this result by explicitly building a translation surface on which $\Sigma_n$ admits a cellular-systolic embedding. If the graph $\Sigma_n$ admits such an embedding with $k\geq 1$ complementary disks, then the Euler characteristic formula gives $$ 1 - n + k = 2 - 2g \implies g = \frac{n + 1 - k}{2}.$$ Therefore, $g$ is maximized when $k = 1$ or $k = 2$, depending on whether $n$ is even or odd, respectively. In either case, the maximum genus is given by $\left\lfloor \frac{n}{2} \right\rfloor$.

To prove (1), it suffices to construct a translation surface of the genus $\left\lfloor \frac{n}{2} \right\rfloor$ on which $\Sigma_n$ admits cellular-systolic embedding. We proceed by considering the following two cases.

Case 1. We assume that $n$ is an even integer, i.e., $n = 2m$, for some $m \in \mathbb{N}$. In this case, we construct the translation surface by identifying opposite parallel sides of a regular $4m$-gon. This yields a translation surface with a single singular point of cone angle $(4m - 2)\pi$ and $n$ systolic connections, one for each side-pair. This gives the required cellular-systolic embedding of $\Sigma_n$. The genus of the resulting surface is $m = \left\lfloor \frac{n}{2} \right\rfloor$. For example, the construction of the translation surface for $\Sigma_4$ is shown in Figure~\ref{figure1}(a).

Case 2. In this case, we assume $n$ is an odd integer, i.e., $n = 2m + 1$, where $m \in \mathbb{N}$. To construct the desired translation surface, we begin with a regular $(2m+1)$-gon $P$, and label its sides by $s_0, s_1, \dots, s_{2m},$ in counterclockwise order. Then, take the reflection of $P$ on the side $s_0$, and denote this reflected polygon by $P'$. Let $s_i'$ denote the image of side $s_i$ under the reflection, for $i = 0, 1, \dots, 2m$.
We identify the sides of $P$ and $P'$ as follows: for each $i = 1, 2, \dots, 2m$, identify $s_i$ with $s_{2m+1 - i}'$ by translation. This yields a translation surface with a single singular point, where all the vertices of the polygons are identified, cone angle $(4m - 2)\pi$. The surface has exactly $2m+1$ systolic connections corresponding to the identified edges. The genus of the resulting surface is $g = m=\left\lfloor \frac{n}{2} \right\rfloor$ and provides a cellular-systolic embedding for $\Sigma_n$. For example, the construction of the translation surface for $\Sigma_5$ is illustrated in Figure~\ref{figure1}(b).

 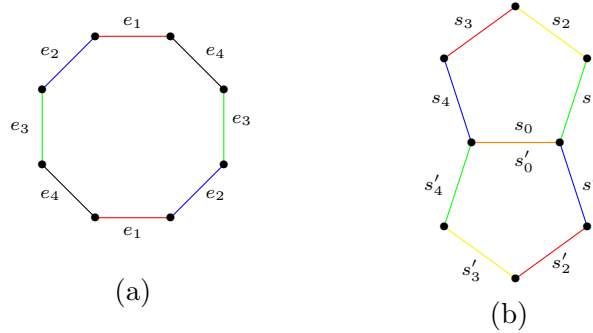
\begin{figure}[htbp]
\begin{center}
\begin{tikzpicture}[scale=1]
\draw[red] (-2,0)--(-1,0);
\draw[blue] (-1,0)--(-0.293,0.707);
\draw[green] (-0.293,0.707)--(-0.293,1.707);
\draw  (-0.293,1.707)--(-1,2.414);
\draw[red] (-2, 2.414)--(-1,2.414);
\draw[blue] (-2,2.414)--(-2.707,1.707);
\draw[green] (-2.707,1.707)--(-2.707,0.707);
\draw (-2.707,0.707)--(-2,0);

\draw (-2,0) node{\tiny$\bullet$};
\draw (-1,0) node{\tiny$\bullet$};
\draw (-0.293,0.707) node{\tiny$\bullet$};
\draw (-0.293,1.707) node{\tiny$\bullet$};
\draw (-2,2.414) node{\tiny$\bullet$};
\draw (-1,2.414) node{\tiny$\bullet$};
\draw (-2.707,1.707) node{\tiny$\bullet$};
\draw (-2.707,0.707) node{\tiny$\bullet$};
\draw (-1.5,-0.2) node{\tiny$e_1$};
\draw (-0.4,0.3) node{\tiny$e_2$};
\draw (-0.05,1.3) node{\tiny$e_3$};
\draw (-0.4,2.2) node{\tiny$e_4$};
\draw (-1.5,2.6) node{\tiny$e_1$};
\draw (-2.6,2.2) node{\tiny$e_2$};
\draw (-3,1.2) node{\tiny$e_3$};
\draw (-2.6,0.3) node{\tiny$e_4$};

\draw (-1.5,-1) node {(a)};

\draw[orange] (3,1)--(4.174,1);
\draw[green] (4.174,1)--(4.538,2.118);
\draw[yellow] (4.538,2.118)--(3.587,2.809);
\draw[red] (3.587,2.809)--(2.636,2.118);
\draw[blue] (2.636,2.118)--(3,1);
\draw[blue] (4.174,1)--(4.538,-.118);
\draw[red] (4.538,-.118)--(3.587,-0.809);
\draw[yellow] (3.587,-0.809)--(2.636,-0.118);
\draw[green] (2.636,-0.118)--(3,1);
\draw (3,1) node{\tiny$\bullet$};
\draw (4.174,1) node{\tiny$\bullet$};
\draw (4.538,2.118) node{\tiny$\bullet$};
\draw (3.587,2.809) node{\tiny$\bullet$};
\draw (2.636,2.118) node{\tiny$\bullet$};
\draw (4.538,-0.118) node{\tiny$\bullet$};
\draw (3.587,-0.809) node{\tiny$\bullet$};
\draw (2.636,-0.118) node{\tiny$\bullet$};
\draw (3.5,-1.3) node {(b)};
\draw (3.7,1.2 ) node {\tiny$s_0$};
\draw (3.7,0.75 ) node {\tiny$s_0^\prime$};
\draw (4.6,1.559 ) node {\tiny$s_1$};
\draw (4.2,2.6 ) node {\tiny$s_2$};
\draw (2.9,2.6) node {\tiny$s_3$};
\draw (2.6,1.559) node {\tiny$s_4$};
\draw (4.6,.441) node {\tiny$s_1^\prime$};
\draw (4.2,-0.6) node {\tiny$s_2^\prime$};
\draw (3,-0.7) node {\tiny$s_3^\prime$};
\draw (2.5,.441) node {\tiny$s_4^\prime$};
\end{tikzpicture}
\end{center}
\caption{Translation Surfaces for (a) $\Sigma_4$ and (b) $\Sigma_5.$}
\label{figure1}
\end{figure}

\item The maximum number of systolic connections on a translation surface of genus $g$ and a single singular point is $3(2g-1)$ (Proposition 5.1 in~\cite{Bo21}).  If $\Sigma_n$ admits a cellular-systolic embedding on a translation surface of genus $g$, then we have $$ n \leq 3(2g - 1)  \implies \frac{n + 3}{6}\leq g.$$ To prove the result, it is sufficient to construct a translation surface of genus $\lceil \frac{n + 3}{6} \rceil$ on which $\Sigma_n$ admits a cellular-systolic embedding. 

For $n = 2, 3$, we have $\lceil \frac{n + 3}{6} \rceil = 1$. The desired translation surfaces for $\Sigma_2$ and $\Sigma_3$ are obtained from the square torus and the equilateral torus (formed by two equilateral triangles), respectively (see Figure~\ref{figure3}). Now, for $n \geq 4$, let $g = \lceil \frac{n + 3}{6} \rceil$. Then $g \geq 2$, and we have $6g - 9 < n \leq 6g - 6.$
Hence, $n \in P_g$ where $P_g = \{6g - 8, 6g - 7, 6g - 6, 6g - 5, 6g - 4, 6g - 3\}.$

Our objective is to construct, for every integer $g \geq 2$, a translation surface of genus $g$ that realises the graph $\Sigma_k$ for each $k \in P_g$.  
In the case of genus $g = 2$, we have $P_2 = \{4, 5, 6, 7, 8, 9\}$. The corresponding translation surfaces of genus $2$ for $\Sigma_k$, with $k \in P_2$, are depicted in Figures~\ref{figure sigma4} and~\ref{figure Sigma8}.

For any $g\geq 3$, to construct translation surfaces for $\Sigma_{6g - 8}, \Sigma_{6g - 7}, \Sigma_{6g - 6}, \dots,$ and $ \Sigma_{6g - 3}$, we take the polygons used in the construction of translation surfaces for $\Sigma_4, \Sigma_5, \Sigma_6, \dots,$ and $\Sigma_9$, respectively, and attach $4(g - 2)$ equilateral triangles. These pieces are glued together as shown in Figures~\ref{figure Sigma(6g-8)} to \ref{figure Sigma(6g-3)}. This gives the desired translation surfaces.
In these constructions, each side of the polygons represents a systolic connection of the surface.
Therefore, for every $k \in P_g$ and for all $g \geq 2$, we obtain translation surfaces of genus $g$ on which $\Sigma_k$ admits a cellular-systolic embedding. 

\begin{figure}[htbp]
\begin{center}
\begin{tikzpicture}[scale=2]
\draw[red] (0,0)--(1,0);
\draw[blue] (1,0)--(1,1);
\draw[red] (1,1)--(0,1);
\draw[blue]  (0,0)--(0,1);

\draw (0,0) node{$\bullet$};
\draw (1,0) node{$\bullet$};
\draw (1,1) node{$\bullet$};
\draw (0,1) node{$\bullet$};
\draw (0.5,-0.1) node{$e_1$};
\draw (0.5,1.1) node{$e_1$};
\draw (-0.15,0.5) node{$e_2$};
\draw (1.15,0.5) node{$e_2$};

\draw[red] (3,0)--(4,0);
\draw[blue] (3,0)--(3.5,0.866);
\draw[green] (3.5,0.866)--(4,0);
\draw[red] (3.5,0.866)--(4.5,0.866);
\draw[blue] (4,0)--(4.5,0.866);
\draw (3,0) node{$\bullet$};
\draw (4,0) node{$\bullet$};
\draw (3.5,0.866) node{$\bullet$};
\draw (4.5,0.866) node{$\bullet$};
\draw (3.5,-0.1) node{$e_1$};
\draw (4,1) node{$e_1$};
\draw (3.1,0.433) node{$e_2$};
\draw (4.4,0.433) node{$e_2$};
\draw (3.9,0.433) node{$e_3$};
\end{tikzpicture}
\caption{Translation surfaces for $\Sigma_2$ and $\Sigma_3$.}
\label{figure3}
\end{center}
\end{figure}
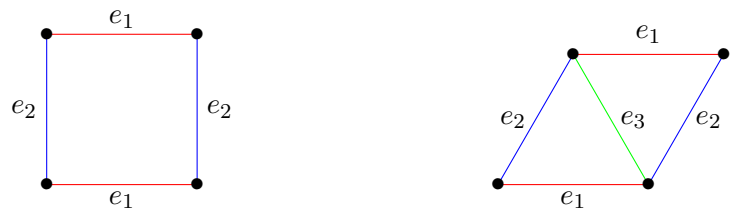

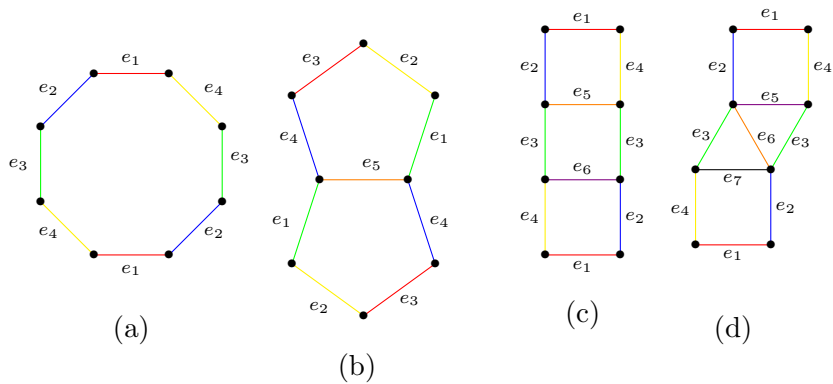
\begin{figure}[htbp]
\begin{center}
\begin{tikzpicture}[scale=1]
\draw[red] (0,0)--(1,0);
\draw[blue] (1,0)--(1.707,0.707);
\draw[green] (1.707,0.707)--(1.707,1.707);
\draw[yellow]  (1.707,1.707)--(1,2.414);
\draw[red] (0, 2.414)--(1,2.414);
\draw[blue] (0,2.414)--(-0.707,1.707);
\draw[green] (-0.707,1.707)--(-0.707,0.707);
\draw[yellow] (-0.707,0.707)--(0,0);
\draw (0,0) node{\tiny$\bullet$};
\draw (1,0) node{\tiny$\bullet$};
\draw (1.707,0.707) node{\tiny$\bullet$};
\draw (1.707,1.707) node{\tiny$\bullet$};
\draw (1,2.414) node{\tiny$\bullet$};
\draw (0,2.414) node{\tiny$\bullet$};
\draw (-0.707,1.707) node{\tiny$\bullet$};
\draw (-0.707,0.707) node{\tiny$\bullet$};
\draw (0.5, -1) node{(a)};
\draw (3.5,-1.5) node{(b)};
\draw (0.5,-0.2) node{\tiny$e_1$};
\draw (1.6,0.3) node{\tiny$e_2$};
\draw (1.95,1.3) node{\tiny$e_3$};
\draw (1.6,2.2) node{\tiny$e_4$};
\draw (0.5,2.6) node{\tiny$e_1$};
\draw (-0.6,2.2) node{\tiny$e_2$};
\draw (-1,1.2) node{\tiny$e_3$};
\draw (-0.6,0.3) node{\tiny$e_4$};

\draw[orange] (3,1)--(4.174,1);
\draw[green] (4.174,1)--(4.538,2.118);
\draw[yellow] (4.538,2.118)--(3.587,2.809);
\draw[red] (3.587,2.809)--(2.636,2.118);
\draw[blue] (2.636,2.118)--(3,1);
\draw[blue] (4.174,1)--(4.538,-.118);
\draw[red] (4.538,-.118)--(3.587,-0.809);
\draw[yellow] (3.587,-0.809)--(2.636,-0.118);
\draw[green] (2.636,-0.118)--(3,1);
\draw (3,1) node{\tiny$\bullet$};
\draw (4.174,1) node{\tiny$\bullet$};
\draw (4.538,2.118) node{\tiny$\bullet$};
\draw (3.587,2.809) node{\tiny$\bullet$};
\draw (2.636,2.118) node{\tiny$\bullet$};
\draw (4.538,-0.118) node{\tiny$\bullet$};
\draw (3.587,-0.809) node{\tiny$\bullet$};
\draw (2.636,-0.118) node{\tiny$\bullet$};
\draw (3.7,1.2 ) node {\tiny$e_5$};
\draw (4.6,1.559 ) node {\tiny$e_1$};
\draw (4.2,2.6 ) node {\tiny$e_2$};
\draw (2.9,2.6) node {\tiny$e_3$};
\draw (2.6,1.559) node {\tiny$e_4$};
\draw (4.6,.441) node {\tiny$e_4$};
\draw (4.2,-0.6) node {\tiny$e_3$};
\draw (3,-0.7) node {\tiny$e_2$};
\draw (2.5,.441) node {\tiny$e_1$};

\draw[red] (6,0)--(7,0);
\draw[blue] (7,0)--(7,1);
\draw[green] (7,1)--(7,2);
\draw[yellow] (7,2)--(7,3);
\draw[red] (7,3)--(6,3);
\draw[blue] (6,3)--(6,2);
\draw[green] (6,2)--(6,1);
\draw[yellow](6,1)--(6,0);
\draw[orange]{} (6,2)--(7,2);
\draw[violet] (6,1)--(7,1);
\draw (6.5,-0.15) node {\tiny$e_1$};
\draw (7.2,0.5) node {\tiny$e_2$};
\draw (7.2,1.5) node {\tiny$e_3$};
\draw (7.2,2.5) node {\tiny$e_4$};
\draw (6.5,3.15) node {\tiny$e_1$};
\draw (5.8,2.5) node {\tiny$e_2$};
\draw (5.8,1.5) node {\tiny$e_3$};
\draw (5.8,0.5) node {\tiny$e_4$};
\draw (6.5,2.15) node {\tiny$e_5$};
\draw (6.5,1.15) node {\tiny$e_6$};

\draw (6,0) node{\tiny$\bullet$};
\draw (7,0) node{\tiny$\bullet$};
\draw (7,1) node{\tiny$\bullet$};
\draw (7,2) node{\tiny$\bullet$};
\draw (7,3) node{\tiny$\bullet$};
\draw (6,3) node{\tiny$\bullet$};
\draw (6,2) node{\tiny$\bullet$};
\draw (6,1) node{\tiny$\bullet$};
\draw (6.5,-0.8) node {(c)};

\draw[yellow] (9.5,2)--(9.5,3);
\draw[red] (9.5,3)--(8.5,3);
\draw[blue] (8.5,3)--(8.5,2);
\draw[violet] (8.5,2)--(9.5,2);
\draw[orange] (8.5,2)--(9,1.134);
\draw[green] (9,1.134)--(9.5,2);
\draw (9,1.134)--(8,1.134);
\draw[green] (8,1.134)--(8.5,2);
\draw[yellow] (8,1.134)--(8,0.134);
\draw[red] (8,0.134)--(9,0.134);
\draw[blue] (9,0.134)--(9,1.134);

\draw (8.5,0) node {\tiny$e_1$};
\draw (9.2,0.634) node {\tiny$e_2$};
\draw (8.08,1.6) node {\tiny$e_3$};
\draw (9.4,1.5) node {\tiny$e_3$};
\draw (9.7,2.5) node {\tiny$e_4$};
\draw (9,3.2) node {\tiny$e_1$};
\draw (8.3,2.5) node {\tiny$e_2$};
\draw (7.8,0.6) node {\tiny$e_4$};
\draw (9,2.1) node {\tiny$e_5$};
\draw (8.95,1.56) node {\tiny$e_6$};
\draw (8.5,1) node {\tiny$e_7$};

\draw (9.5,2) node{\tiny$\bullet$};
\draw (9.5,3) node{\tiny$\bullet$};
\draw (8.5,3) node{\tiny$\bullet$};
\draw (8.5,2) node{\tiny$\bullet$};
\draw (9,1.134) node{\tiny$\bullet$};
\draw (8,1.134) node{\tiny$\bullet$};
\draw (9,0.134) node{\tiny$\bullet$};
\draw (8,0.134) node{\tiny$\bullet$};
\draw (8.5,-1) node {(d)};

\end{tikzpicture}

\caption{Translation surfaces for (a) $\Sigma_4$, (b) $\Sigma_5$, (c) $\Sigma_6$ and (d) $\Sigma_7.$}
\label{figure sigma4}
\end{center}
\end{figure}

\begin{figure}[htbp]
\begin{center}
\rotatebox{0}{
\begin{tikzpicture}[scale=1]
\draw[] (0,1)--(0.5,0.134);
\draw[blue] (0.5,.134)--(1,1);
\draw[red] (0.5,0.134)--(-0.5,0.134);
\draw[yellow] (-0.5,0.134)--(0,1);
\draw[green] (1,1)--(1,2);
\draw[blue] (0,2)--(0.5,2.866);
\draw[violet] (0.5,2.866)--(1,2);
\draw[red] (0.5,2.866)--(1.5,2.866);
\draw[yellow] (1.5,2.866)--(1,2);
\draw[green] (0,2)--(0,1);
\draw[orange]{} (0,2)--(1,2);
\draw[violet] (0,1)--(1,1);
\draw (0,0) node {\tiny$e_1$};
\draw (0.9,0.5) node {\tiny$e_2$};
\draw (1.2,1.5) node {\tiny$e_3$};
\draw (1.5,2.433) node {\tiny$e_4$};
\draw (1,3) node {\tiny$e_1$};
\draw (0,2.433) node {\tiny$e_2$};
\draw (-0.2,1.5) node {\tiny$e_3$};
\draw (-0.5,0.6) node {\tiny$e_4$};
\draw (0.95,2.433) node {\tiny$e_5$};
\draw (0.5,2.1) node {\tiny$e_6$};
\draw (0.5,1.1) node {\tiny$e_7$};
\draw (0.15,0.4) node {\tiny$e_8$};

\draw (0,1) node{\tiny$\bullet$};
\draw (0.5,0.134) node{\tiny$\bullet$};
\draw (1,1) node{\tiny$\bullet$};
\draw (-0.5,0.134) node{\tiny$\bullet$};
\draw (0,2) node{\tiny$\bullet$};
\draw (1,2) node{\tiny$\bullet$};
\draw (0.5,2.866) node{\tiny$\bullet$};
\draw (1.5,2.866) node{\tiny$\bullet$};
\draw (0,-0.8) node {(e)};

\draw[orange] (3,1.732)--(3.5,0.866);
\draw[green] (3.5,0.866)--(4,1.732);
\draw[] (3.5,0.866)--(2.5,0.866);
\draw[green] (2.5,0.866)--(3,1.714);
\draw[magenta] (2.5,0.866)--(3,0);
\draw[blue] (3,0)--(3.5,0.866);
\draw[red] (3,0)--(2,0);
\draw[yellow] (2,0)--(2.5,0.866);
\draw[blue] (3,1.714)--(3.5,2.58);
\draw[violet] (3.5,2.58)--(4,1.714);
\draw[red] (3.5,2.58)--(4.5,2.58);
\draw[yellow] (4.5,2.58)--(4,1.714);
\draw[cyan] (3,1.714)--(4,1.714);

\draw (2.5,-0.2) node {\tiny$e_1$};
\draw (3.4,0.3) node {\tiny$e_2$};
\draw (4,1.299) node {\tiny$e_3$};
\draw (4.5,2.2) node {\tiny$e_4$};
\draw (4,2.7) node {\tiny$e_1$};
\draw (3,2.2) node {\tiny$e_2$};
\draw (2.5,1.3) node {\tiny$e_3$};
\draw (2,0.433) node {\tiny$e_4$};
\draw (3.9,2.2) node {\tiny$e_5$};
\draw (3.5,1.6) node {\tiny$e_6$};
\draw (3.15,1.2) node {\tiny$e_7$};
\draw (3,0.7) node {\tiny$e_8$};
\draw (2.6,0.3) node {\tiny$e_9$};

\draw (3,1.714) node{\tiny$\bullet$};
\draw (3.5,0.866) node{\tiny$\bullet$};
\draw (4,1.714) node{\tiny$\bullet$};
\draw (2.5,0.866) node{\tiny$\bullet$};
\draw (3,0) node{\tiny$\bullet$};
\draw (3.5,0.866) node{\tiny$\bullet$};
\draw (2,0) node{\tiny$\bullet$};
\draw (3.5,2.58) node{\tiny$\bullet$};
\draw (4.5,2.58) node{\tiny$\bullet$};
\draw (2.5,-1) node {(f)};

\end{tikzpicture}
}
\caption{Translation surfaces for (e) $\Sigma_8$ and (f) $\Sigma_9.$}
\label{figure Sigma8}
\end{center}
\end{figure}
\begin{figure}[htbp]
\begin{center}
\rotatebox{0}{
\begin{tikzpicture}[scale=1]
\draw[] (0,0)--(1,0);
\draw[blue] (1,0)--(1.707,0.707);
\draw[green] (1.707,0.707)--(1.707,1.707);
\draw[red]  (1.707,1.707)--(1,2.414);
\draw[] (0, 2.414)--(1,2.414);
\draw[blue] (0,2.414)--(-0.707,1.707);
\draw[green] (-0.707,1.707)--(-0.707,0.707);
\draw[red] (-0.707,0.707)--(0,0);
\draw (1.5,0.3) node{\tiny$e_1$};
\draw (1.9,1.1) node{\tiny$e_2$};
\draw (1.2,2) node{\tiny$e_3$};
\draw (-0.5,2.2) node{\tiny$e_1$};
\draw (-0.9,1.1) node{\tiny$e_2$};
\draw (-0.1,0.35) node{\tiny$e_3$};
\draw (0,2.414)--(0.5,3.28)--(1,2.414);
\draw (0.5,3.28)--(1.5,3.28)--(1,2.414);
\draw (1,2.414)--(2,2.414)--(1.5,3.28);
\draw (1.5,3.28)--(2.5,3.28)--(2,2.414);
\draw[dashed] (2,2.414)--(3,2.414);
\draw[dashed] (2.5,3.28)--(3.5,3.28);
\draw (3.5,3.28)--(4.5,3.28)--(4,2.414)--(3.5,3.28);
\draw (3.5,3.28)--(3,2.414)--(4,2.414);
\draw (1,3.4) node{\tiny$a_1$};
\draw (2,3.4) node{\tiny$a_2$};
\draw (4,3.4) node{\tiny$a_{g-2}$};
\draw (1.7,2.25) node{\tiny$b_1$};
\draw (3.5,2.28) node{\tiny$b_{g-3}$};

\draw (0,0)--(0.5,-0.866)--(1,0);
\draw (0.5,-0.866)--(-0.5,-0.866)--(0,0);
\draw (0,0)--(-1,0)--(-0.5,-0.866);
\draw(-1,0)--(-1.5,-0.866)--(-0.5,-0.866);
\draw (-1,0)--(-2,0)--(-1.5,-0.866);
\draw (0.1,2.914) node{\tiny$c$};
\draw (0.9,-0.5) node {\tiny$c$};
\draw (4.3,2.7) node{\tiny$d$};
\draw (-4.4,-0.366) node{\tiny$d$};
\draw[dashed] (-2,0)--(-3,0);
\draw[dashed] (-1.5,-0.866)--(-3.5,-0.866);
\draw (-3.5,-0.866)--(-4.5,-0.866)--(-4,0)--(-3.5,-0.866);
\draw (-3,0)--(-4,0);
\draw (-3,0)--(-3.5,-0.866);
\draw (0,-1) node{\tiny$a_1$};
\draw (-1,-1) node {\tiny$a_2$};
\draw (-4,-1) node {\tiny$a_{g-2}$};
\draw (-0.5,0.2) node {\tiny$b_1$};
\draw (-1.5,0.2) node {\tiny$b_2$};
\draw (-3.5,0.2) node {\tiny$b_{g-3}$};
\end{tikzpicture}
}
\caption{Translation surface for $\Sigma_{6g-8}$}
\label{figure Sigma(6g-8)}
\end{center}
\end{figure}

\begin{figure}[htbp]
\begin{center}
\begin{tikzpicture}[scale=1]
\draw[orange] (0,0)--(1.174,0);
\draw[green] (1.174,0)--(1.538,1.118);
\draw[] (1.538,1.118)--(0.587,1.809);
\draw[red] (0.587,1.809)--(-0.364,1.118);
\draw[blue] (-0.364,1.118)--(0,0);
\draw[blue] (1.174,0)--(1.538,-1.118);
\draw[red] (1.538,-1.118)--(0.587,-1.809);
\draw[] (0.587,-1.809)--(-0.364,-1.118);
\draw[green] (-0.364,-1.118)--(0,0);

\draw (1.3,-1.6) node{\tiny$e_1$};
\draw (1.56, -0.55) node{\tiny$e_2$};
\draw (0.1,1.7) node{\tiny$e_1$};
\draw (-0.4,0.55) node{\tiny$e_2$};
\draw (-.4,-0.55) node{\tiny$e_3$};
\draw (1.55,0.5) node{\tiny$e_3$};

\draw (0.587,1.809)--(1.5003,2.2161)--(1.538,1.118)--(2.309,1.634)--(1.5003,2.2161);
\draw[dashed] (2.309,1.634)--(3.118,1.041);
\draw[dashed] (1.538,1.118)--(2.205,0.634);
\draw (3.118,1.041)--(2.205,0.634);
\draw (2.205,0.634)--(3.014,0.046);
\draw (3.118,1.041)--(3.928,0.454)--(3.014,0.046)--(3.118,1.041);

\draw(-0.364,-1.118)--(-0.467,-2.112)--(0.587,-1.809);
\draw (-0.467,-2.112)--(-1.277,-1.525)--(-0.364,-1.118);
\draw[dashed] (-0.364,-1.118)--(-1.173,-0.530);
\draw[dashed] (-1.277,-1.525)--(-2.086,-0.937);
\draw (-1.173,-0.530)--(-2.086,-0.937);
\draw (-1.173,-0.530)--(-1.982,0.056);
\draw (-2.086,-0.937)--(-2.895,-0.350);
\draw (-2.895,-0.350)--(-1.982,0.056)--(-2.086,-0.937);

\draw (2,2) node{\tiny$a_1$};
\draw (-0.9,-2) node{\tiny$a_1$};
\draw (3.75,0.9) node{\tiny$a_{g-2}$};
\draw (-2.8,-0.8) node {\tiny$a_{g-2}$};
\draw (2.5,0.1) node {\tiny$b_{g-3}$};
\draw (-1.4,0) node {\tiny$b_{g-3}$};
\draw (1,2.2) node {\tiny$c$};
\draw (0,-2.2) node {\tiny$c$};
\draw (3.6,0.15) node {\tiny$d$};
\draw (-2.6,0) node {\tiny$d$};

\end{tikzpicture}

\caption{Translation surface for $\Sigma_{6g-7}$}
\label{figure Sigma(6g-7)}
\end{center}
\end{figure}

\begin{figure}[htbp]
\begin{center}
\begin{tikzpicture}[scale=1]
\draw[] (0,0)--(1,0);
\draw[blue] (1,0)--(1,1);
\draw[green] (1,1)--(1,2);
\draw[yellow] (1,2)--(1,3);
\draw[] (1,3)--(0,3);
\draw[blue] (0,3)--(0,2);
\draw[green] (0,2)--(0,1);
\draw[yellow](0,1)--(0,0);
\draw[orange]{} (0,2)--(1,2);
\draw[violet] (0,1)--(1,1);

\draw (1.2,.5) node{\tiny$e_1$};
\draw (1.2,1.5) node{\tiny$e_2$};
\draw (1.2,2.4) node{\tiny$e_3$};
\draw (-0.2,.5) node{\tiny$e_3$};
\draw (-0.2,1.5) node{\tiny$e_2$};
\draw (-0.2,2.4) node{\tiny$e_1$};

\draw (0,0)--(0.5,-0.866)--(1,0);
\draw (0.5,-0.866)--(-0.5,-0.866)--(0,0);
\draw (0,0)--(-1,0)--(-0.5,-0.866);
\draw (-1,0)--(-1.5,-0.866)--(-0.5,-0.866);
\draw[dashed] (-1,0)--(-2,0);
\draw[dashed] (-1.5,-0.866)--(-2.5,-0.866);
\draw (-3,0)--(-2.5,-0.866)--(-3.5,-0.866)--(-3,0);
\draw (-3,0)-- (-2,0)--(-2.5,-0.866);

\draw (0,3)--(0.5,3.866)--(1,3);
\draw (0.5,3.866)--(1.5,3.866)--(1,3);
\draw (1,3)--(2,3)--(1.5,3.866)--(2.5,3.866)--(2,3);
\draw[dashed] (2.5,3.866)--(3.5,3.866);
\draw[dashed] (2,3)--(3,3);
\draw (3.5,3.866)--(4,3)--(4.5,3.866)--(3.5,3.866);
\draw (4,3)--(3,3)--(3.5,3.866);
\draw (1,4) node {\tiny$a_1$};
\draw (2,4) node {\tiny$a_2$};
\draw (4,4) node {\tiny$a_{g-2}$};

\draw (0,-1) node{\tiny$a_1$};
\draw (-1,-1) node{\tiny$a_2$};
\draw (-3,-1) node{\tiny$a_{g-2}$};
\draw (1.5,2.8) node{\tiny$b_1$};
\draw (3.5,2.8) node{\tiny$b_{g-3}$};
\draw (-0.5,0.2) node {\tiny$b_1$};
\draw (-2.5,0.2) node{\tiny$b_{g-3}$};
\draw (0.1,3.5) node {\tiny$c$};
\draw (0.8,-.5) node {\tiny$c$};
\draw (4.35,3.4) node {\tiny$d$};
\draw (-3.38,-.35) node {\tiny$d$};
\end{tikzpicture}
\caption{ Translation surface for $\Sigma_{6g-6}$}
\label{figure sigma(6g-6)}
\end{center}
\end{figure}

\begin{figure}[htbp]
\begin{center}
\begin{tikzpicture}[scale=1]
\draw[red] (1,2)--(1,3);
\draw[] (1,3)--(0,3);
\draw[blue] (0,3)--(0,2);
\draw[violet] (0,2)--(1,2);
\draw[orange] (0,2)--(0.5,1.134);
\draw[green] (0.5,1.134)--(1,2);
\draw (0.5,1.134)--(-.5,1.134);
\draw[green] (-.5,1.134)--(0,2);
\draw[red] (-0.5,1.134)--(-0.5,0.134);
\draw[] (-0.5,0.134)--(0.5,0.134);
\draw[blue] (0.5,0.134)--(0.5,1.134);

\draw (1.2,2.4) node{\tiny$e_3$};
\draw (-.7,0.4) node{\tiny$e_3$};
\draw (-0.4,1.7) node{\tiny$e_2$};
\draw (1,1.5) node{\tiny$e_2$};
\draw (-0.2,2.6) node{\tiny$e_1$};
\draw (0.7,0.5) node{\tiny$e_1$};

\draw (0,3)--(0.5,3.866)--(1.5,3.866)--(1,3)--(0.5,3.866);
\draw (1,3)--(2,3)--(1.5,3.866)--(2.5,3.866)--(2,3);
\draw[dashed] (2,3)--(4,3);
\draw[dashed] (2.5,3.866)--(4.5,3.866);
\draw (4.5,3.866)--(5.5,3.866)--(5,3)--(4,3)--(4.5,3.866)--(5,3);
\draw (0.5,0.134)--(0,-0.732)--(-1,-0.732)--(-0.5,0.134)--(0,-0.732);
\draw (-0.5,0.134)--(-1.5,0.134)--(-1,-0.732)--(-2,-0.732)--(-1.5,0.134);
\draw[dashed] (-1.5,0.134)--(-3.5,0.134);
\draw[dashed] (-2,-0.732)--(-4,-0.732);
\draw (-4,-0.732)--(-3.5,0.134)--(-4.5,0.134)--(-4,-0.732)--(-5,-0.732)--(-4.5,0.134);

\draw (0.9,4) node {\tiny$a_1$};
\draw (1.9,4) node {\tiny$a_2$};
\draw (4.9,4) node {\tiny$a_{g-2}$};
\draw (4.7,2.8) node {\tiny$b_{g-3}$};
\draw (1.6,2.8) node {\tiny$b_1$};
\draw (0.1,3.5) node{\tiny$c$};
\draw (5.5,3.5) node{\tiny$d$};
\draw (0.3,-0.5) node{\tiny$c$};
\draw (-5,-0.3) node{\tiny$d$};

\draw (-0.5,-1) node {\tiny$a_1$};
\draw (-1.5,-1) node {\tiny$a_2$};
\draw (-4.5,-1) node {\tiny$a_{g-2}$};
\draw (-4.3,0.35) node {\tiny$b_{g-3}$};
\draw (-1.3,0.35) node {\tiny$b_{1}$};

\end{tikzpicture}
\caption{Translation surface for $\Sigma_{6g-5}$}
\label{figure Sigma 6g-5}
\end{center}
\end{figure}

\begin{figure}[htbp]
\begin{center}
\begin{tikzpicture}[scale=1]
\draw[] (0,1)--(0.5,0.134);
\draw[blue] (0.5,.134)--(1,1);
\draw[] (0.5,0.134)--(-0.5,0.134);
\draw[red] (-0.5,0.134)--(0,1);
\draw[green] (1,1)--(1,2);
\draw[blue] (0,2)--(0.5,2.866);
\draw[violet] (0.5,2.866)--(1,2);
\draw[] (0.5,2.866)--(1.5,2.866);
\draw[red] (1.5,2.866)--(1,2);
\draw[green] (0,2)--(0,1);
\draw[orange]{} (0,2)--(1,2);
\draw[violet] (0,1)--(1,1);

\draw (0.5,2.866)--(1,3.866)--(2,3.866)--(1.5,2.866)--(1,3.866);
\draw (1.5,2.866)--(2.5,2.866)--(3,3.866)--(2,3.866)--(2.5,2.866);
\draw[dashed] (3,3.866)--(5,3.866);
\draw[dashed] (2.5,2.866)--(4.5,2.866);
\draw (5,3.866)--(4.5,2.866)--(5.5,2.866)--(6,3.866)--(5,3.866)--(5.5,2.866);

\draw (0.5,0.134)--(0,-0.732)--(-1,-0.732)--(-0.5,0.134)--(0,-0.732);
\draw (-0.5,0.134)--(-1.5,0.134)--(-1,-0.732)--(-2,-0.732)--(-1.5,0.134);
\draw[dashed] (-1.5,0.134)--(-3.5,0.134);
\draw[dashed] (-2,-0.732)--(-4,-0.732);
\draw (-4,-0.732)--(-3.5,0.134)--(-4.5,0.134)--(-4,-0.732)--(-5,-0.732)--(-4.5,0.134);

\draw (1.5,4) node{\tiny$a_1$};
\draw (2.5,4) node{\tiny$a_2$};
\draw (5.5,4) node{\tiny$a_{g-2}$};

\draw (5.1,2.7) node{\tiny$b_{g-3}$};
\draw (2.1,2.7) node{\tiny$b_{1}$};
\draw (-1,0.3) node{\tiny$b_{1}$};
\draw (-4,0.3) node{\tiny$b_{g-3}$};

\draw (-0.5,-0.9) node{\tiny$a_{1}$};
\draw (-1.5,-0.9) node{\tiny$a_{2}$};
\draw (-4.5,-0.9) node{\tiny$a_{g-2}$};

\draw (0.4,-0.4) node{\tiny$c$};
\draw (-5,-0.3) node{\tiny$d$};
\draw (0.6,3.5) node{\tiny$c$};
\draw (5.9,3.4) node{\tiny$d$};

\draw (0.9,0.5) node{\tiny$e_1$};
\draw (1.2,1.5) node{\tiny$e_2$};
\draw (1.5,2.5) node{\tiny$e_3$};

\draw (0.1,2.5) node{\tiny$e_1$};
\draw (-0.2,1.5) node{\tiny$e_2$};
\draw (-0.5,0.6) node{\tiny$e_3$};
\end{tikzpicture}
\caption{Translation surface for $\Sigma_{6g-4}.$}
\label{figure8}
\end{center}
\end{figure}

\begin{figure}[htbp]
\begin{center}
\rotatebox{0}{
\begin{tikzpicture}[scale=1]
\draw[orange] (0,2)--(0.5,1.134);
\draw[green] (0.5,1.134)--(1,2);
\draw[] (0.5,1.134)--(-0.5,1.134);
\draw[green] (-0.5,1.134)--(0,2);
\draw[magenta] (-0.5,1.134)--(0,0.268);
\draw[blue] (0,0.268)--(0.5,1.134);
\draw  (0,0.268)--(-1,0.268);
\draw[yellow] (-1,0.268)--(-0.5,1.134);
\draw[blue] (0,2)--(0.5,2.866);
\draw[violet] (0.5,2.866)--(1,2);
\draw[] (0.5,2.866)--(1.5,2.866);
\draw[yellow] (1.5,2.866)--(1,2);
\draw[cyan] (0,2)--(1,2);

\draw (0.4,0.65) node{\tiny$e_1$};
\draw (1,1.5) node{\tiny$e_2$};
\draw (1.5,2.43) node{\tiny$e_3$};
\draw (0.1,2.5) node{\tiny$e_1$};
\draw (-0.4,1.7) node{\tiny$e_2$};
\draw (-1,0.65) node{\tiny$e_3$};

\draw (0.5,2.866)--(1,3.732)--(1.5,2.866);
\draw (1,3.732)--(2,3.732)--(1.5,2.866);
\draw (1.5,2.866)--(2.5,2.866)--(2,3.732);
\draw (2.5,2.866)--(3,3.732)--(2,3.732);
\draw[dashed] (2.5,2.866)--(3.5,2.866);
\draw[dashed] (3,3.732)--(4,3.732);
\draw (4,3.732)--(4.5,2.866)--(5,3.732)--(4,3.732);
\draw (4.5,2.866)--(3.5,2.866)--(4,3.732);

\draw (-1,0.268)--(-0.5,-0.598)--(0,0.268);
\draw (-1,0.268)--(-1.5,-.598)--(-0.5,-.598);
\draw (-1,0.268)--(-2,0.268)--(-1.5,-.598);
\draw (-2,0.268)--(-2.5,-.598)--(-1.5,-.598);
\draw[dashed] (-2,0.268)--(-3,0.268);
\draw[dashed] (-2.5,-.598)--(-3.5,-.598);
\draw (-3.5,-.598)--(-4,.268)--(-4.5,-.598)--(-3.5,-.598);
\draw (-4,0.268)--(-3,0.268)--(-3.5,-.598);

\draw (1.5,3.866) node{\tiny$a_1$};
\draw (2.5,3.866) node{\tiny$a_2$};
\draw (4.5,3.866) node{\tiny$a_{g-2}$};
\draw (-1,-0.8) node{\tiny$a_1$};
\draw (-2,-0.8) node{\tiny$a_2$};
\draw (-4,-0.8) node{\tiny$a_{g-2}$};

\draw (0.65,3.366) node{\tiny$c$};
\draw (-0.1,-0.2) node{\tiny$c$};

\draw (4.8,3.1) node{\tiny$d$};
\draw (-4.4,-0.1) node{\tiny$d$};

\draw (2,2.7) node{\tiny$b_1$};
\draw (4,2.7) node{\tiny$b_{g-3}$};
\draw (-1.5,0.45) node{\tiny$b_1$};
\draw (-3.5,0.45) node{\tiny$b_{g-3}$};
\end{tikzpicture}
}
\caption{Translation surface for $\Sigma_{6g-3}.$}
\label{figure Sigma(6g-3)}
\end{center}
\end{figure}

\item
Let $g$ be an integer that satisfies $\left\lceil \frac{n+3}{6} \right\rceil\le g\le\left\lfloor \frac{n}{2} \right\rfloor$. We aim to show that \( \Sigma_n \) admits cellular-systolic embedding on a translation surface of genus $g$.
To construct such a translation surface, we require \( (n - 2g + 1) \) polygonal faces. This follows from the Euler characteristic formula:
$$\chi = V - E + F = 2 - 2g,$$
where \( V \), \( E \), and \( F \) denote the numbers of vertices, edges, and faces, respectively. Now, \( \Sigma_n \) has \( 1 \) vertex and an $n$ edges to yield the desired genus. From that we find \( F \) = \(n-2g+1\).
We consider two cases based on the parity of \( (n - 2g + 1) \):

Case 1. Assume \( n - 2g + 1 = 2k + 1 \), for some \( k \in \mathbb{N} \). In this case, we construct the translation surface using \( 2k \) equilateral triangles and one regular \( (2n - 6k) \)-gon. These polygons are glued together placing \( k \) equilateral triangles on each of the two opposite ends of the \( (2n - 6k) \)-gon, as shown in Figure~\ref{intermediate genus} $(i)$. This configuration yields a translation surface of genus \( g \) that supports a cellular and systolic embedding of \( \Sigma_n \).

Case 2. Assume \( n - 2g + 1 = 2k \), for some \( k \in \mathbb{N} \). In this case, we use \( 2k - 2 \) equilateral triangles and two regular \( (n - 6k + 6) \)-gons. These polygons are glued to form the desired surface, as illustrated in Figure~\ref{intermediate genus} $(ii)$. Again, the resulting translation surface supports a cellular-systolic embedding of \( \Sigma_n \).

In both cases, the number of polygonal faces is exactly \( n - 2g + 1 \), and the resulting surface has genus \( g \), as required by the Euler characteristic. The embeddings are constructed to satisfy the cellular and systolic conditions. Hence, for every integer \( g \) such that $\left\lceil \frac{n+3}{6} \right\rceil\le g\le\left\lfloor \frac{n}{2} \right\rfloor$, such a translation surface exists.
\end{enumerate}
\end{proof}

\begin{figure}[htbp]
\begin{center}
\rotatebox{0}{
\begin{tikzpicture}[scale=1.2]
\draw[blue] (0,1)--(-0.5,0.866);
\draw[green] (-0.5,0.866)--(-0.866,0.5);
\draw[brown] (-0.866,-0.5)--(-0.5,-0.866);
\draw(-0.5,-0.866)--(0,-1);
\draw[dashed] (-0.866,0.5)--(-0.866,-0.5);
\draw[dashed] (0.866,-0.5)--(0.866,0.5);
\draw[blue] (0,-1)--(0.5,-0.866);
\draw[green] (0.5,-0.866)--(0.866,-0.5);
\draw[brown] (0.866,0.5)--(0.5,0.866);

\draw (-0.3,1.05) node{\tiny$e_1$};
\draw (-0.8,0.8) node{\tiny$e_2$};
\draw (0.3,-1.1) node{\tiny$e_1$};
\draw (0.8,-0.8) node{\tiny$e_2$};
\draw (-0.2,-0.7) node{\tiny\tiny$e_{n-3k-1}$};
\draw (0.2,0.6) node{\tiny\tiny$e_{n-3k-1}$};

\draw[] (0.5,0.866)--(0,1);
\draw[] (0,1)--(0.365,1.365)--(0.5,0.866);
\draw (0.365,1.365)--(0.864,1.231)--(0.5,0.866);
\draw (-0.5,-0.866)--(-0.366,-1.366)--(0,-1);
\draw (-0.5,-0.866)--(-0.865,-1.231)--(-0.366,-1.366);
\draw (0.6,1.5) node{\tiny$a_1$};
\draw (-0.6,-1.5) node {\tiny$a_1$};
\draw[dashed] (0.5,0.866)--(1.498,0.598);
\draw[dashed] (0.864,1.231)--(1.363,1.097);
\draw (1.363,1.097)--(1.498,0.598)--(1.863,0.963)--(1.363,1.097);
\draw (1.7, 1.2) node{\tiny$a_{k-2}$};
\draw (-1.7, -1.2) node{\tiny$a_{k-2}$};
\draw (1.8,0.79) node{\tiny$c$};
\draw (0.05,1.3) node{\tiny$d$};
\draw (-1.8,-0.79) node{\tiny$c$};
\draw (-0.2,-1.3) node{\tiny$d$};
\draw (-0.6,-2) node {$(i)$};
\draw[dashed] (-0.5,-0.866)--(-1.498,-0.598);
\draw[dashed] (-0.864,-1.231)--(-1.363,-1.097);
\draw (-1.363,-1.097)--(-1.498,-0.598)--(-1.863,-0.963)--(-1.363,-1.097);
\draw (4.623,0.781)--(3.788,0.947);

\draw[red] (3.788,0.947)--(3.1,0.434);
\draw[dashed] (3.1,0.434)--(3.099,-0.433);
\draw[blue] (3.099,-0.433)--(3.777,-0.974);
\draw[cyan] (3.777,-0.974)--(4.622,-0.782);
\draw[yellow] (4.622,-0.782)--(4.999,0.001);
\draw[green] (4.999,0.001)--(4.623,0.781);
\draw[red] (5.295,-2.190)--(4.606,-2.703);
\draw (4.606,-2.703)--(3.775,-2.533);
\draw[green]  (3.775,-2.533)--(3.4,-1.753);
\draw[yellow] (3.4,-1.753)--(3.777,-0.974);
\draw[dashed] (5.298,-1.325)--(5.295,-2.190);
\draw[blue] (5.298,-1.325)--(4.622,-0.782);

\draw (3.4,0.85) node{\tiny$e_1$};
\draw (5,-2.6) node{\tiny$e_1$};
\draw (5.3,0.37) node{\tiny$e_{n-6k+4}$};
\draw (3.1,-2.2) node{\tiny$e_{n-6k+4}$};
\draw (5.3,-0.4) node{\tiny$e_{n-6k+3}$};
\draw (3,-1.4) node{\tiny$e_{n-6k+3}$};
\draw (2.9,-0.7) node{\tiny$e_{n-6k+2}$};
\draw (5.65,-1.1) node{\tiny$e_{n-6k+2}$};
\draw (3.778,0.947)--(4.340,1.751)--(4.623,0.781)--(5.171,1.621)--(4.340,1.751);
\draw (3.775,-2.533)--(4.043,-3.467)--(4.606,-2.703);
\draw (3.775,-2.533)--(3.213,-3.337)--(4.043,-3.467);
\draw[dashed] (4.623,0.781)--(6.312,0.449);
\draw[dashed] (5.171,1.621)--(6.015,1.455);
\draw (6.015, 1.455)--(6.860,1.289)--(6.312,0.449)--(6.015,1.455);
\draw[dashed] (3.775,-2.533)--(2.086,-2.201);
\draw[dashed] (3.213,-3.337)--(2.369, -3.171);
\draw (2.369,-3.171)--(2.086,-2.201)--(1.524,-3.005)--(2.369,-3.171);
\draw (6.4375, 1.6) node {\tiny$a_{k-3}$};
\draw (1.9465,-3.3) node {\tiny$a_{k-3}$};
\draw (6.75,0.869) node {\tiny$c$};
\draw (1.64, -2.603) node {\tiny$c$};
\draw (4.01,1.5) node {\tiny$d$};
\draw (4.45,-3.085) node {\tiny$d$};
\draw (4.8,1.9) node {\tiny$a_1$};
\draw (3.6,-3.6) node{\tiny$a_1$};
\draw (3.6,-4) node{$(ii)$};
\end{tikzpicture}
}
\caption{} 
\label{intermediate genus}
\end{center}
\end{figure}


As of now, we have identified that the graphs $\Sigma_n$, for $n \geq 2$, admit cellular-systolic embedding on translation surfaces. A natural question arises as to whether there exist other classes of graphs with more than one vertex that admit a cellular-systolic embedding on some translation surface. We answer this affirmatively and find a rich collection of non-simple graphs. In particular, we identify two special classes as follows:

   \begin{enumerate}
       
    \item $\mathcal{P} = \left\{\, G_{n,m} = (V_n, E_n) \mid n \geq 2,\; m \in \mathbb{N} \right\}$, where $V_n = \{ v_0, v_1, \dots, v_{n-1} \}$ and $ E_n = \{ e_{ij} \mid 0 \leq i \leq n-1,\; 1 \leq j \leq 2m \}$. Note that $e_{ij}$ are the $2m$ parallel edges between~$v_i$ and~$v_{i+1}$ (indices of $v'$s taken modulo $n$). For example, see Figure~\ref{P,Q}$(a)$ for $G_{4,1}$.

\item $\mathcal{Q} = \left\{\, G_{n,m}^\prime = (V_n, E_n) \;\mid \; n \geq 2,\; m \in \mathbb{N}\right\}$, where $V_n = \left\{ v_0, v_1, \dots, v_{n-1}\right \}$ and $E_n = L \cup P$; $L = \{\, \ell_{ij} \mid 0 \leq i \leq n-1,\; 1 \leq j \leq m \}$ and  
$P = \{\, e_{ij} \mid 0 \leq i \leq n-1,\; 1 \leq j \leq m \}$,  such that each $\ell_{ij}$ is a loop at vertex $v_i$ and each $e_{ij}$ connects vertex $v_i$ to vertex $v_{i+1}$ ( indices taken modulo $n$). That is, every vertex $v_i$ has exactly $m$ loops, and every pair  $(v_i, v_{i+1})$ is connected by $m$ parallel edges. In Figure~\ref{P,Q}$(b)$, $G_{4,2}^\prime$ is depicted.

\end{enumerate} 

\begin{figure}[htbp]
\begin{center}
\begin{tikzpicture}[scale=1]
\draw (0,0)--(2,0)--(2,1)--(0,1)--(0,0);
\draw (0,0)..controls(1,-0.3)..(2,0);
\draw (2,0)..controls(1.7,0.5)..(2,1);
\draw (2,1)..controls(1,1.3)..(0,1);
\draw (0,1)..controls(0.3,0.5)..(0,0);

\draw (0,0) node{\tiny$\bullet$};
\draw[red] (2,0) node{\tiny$\bullet$};
\draw[blue] (2,1) node{\tiny$\bullet$};
\draw [green] (0,1) node{\tiny$\bullet$};

\draw (4,0)--(6,0)--(6,1)--(4,1)--(4,0);
\draw (4,0)..controls(5,-0.3)..(6,0);
\draw (6,0)..controls(5.7,0.5)..(6,1);
\draw (6,1)..controls(5,1.3)..(4,1);
\draw (4,1)..controls(4.3,0.5)..(4,0);

\draw (4,0) node{\tiny$\bullet$};
\draw[red] (6,0) node{\tiny$\bullet$};
\draw[blue] (6,1) node{\tiny$\bullet$};
\draw [green] (4,1) node{\tiny$\bullet$};

\draw[] (4,0) arc
[
start angle=90,
end angle=450,
x radius=0.2,
y radius=0.4,
];

\draw[] (4,0) arc
[
start angle=89,
end angle=450,
x radius=0.3,
y radius=0.5,
];

\draw[] (6,0) arc
[
start angle=90,
end angle=450,
x radius=0.2,
y radius=0.4,
];

\draw[] (6,0) arc
[
start angle=90,
end angle=450,
x radius=0.3,
y radius=0.5,
];

\draw[] (6,1) arc
[
start angle=270,
end angle=630,
x radius=0.2,
y radius=0.4,
];

\draw (6,1) arc
[
start angle=270,
end angle=630,
x radius=0.3,
y radius=0.5,
];

\draw (4,1) arc
[
start angle=270,
end angle=630,
x radius=0.2,
y radius=0.4,
];

\draw (4,1) arc
[
start angle=270,
end angle=630,
x radius=0.3,
y radius=0.5,
];

\draw (0.8, -1) node {$(a)$}; 
\draw (5, -1) node {$(b)$}; 
\end{tikzpicture}
\caption{$(a)~G_{4,1}$ and $(b)~G_{4,2}^\prime.$}
\label{P,Q}
\end{center}
\end{figure}

   These graphs are nothing but $4m$-regular multigraphs. In the following result, we answer the above question.

\begin{theorem}
Let $G\in \mathcal{P} \sqcup \mathcal{Q}.$ Then $G$ admits a cellular-systolic embedding on some translation surface.   
\end{theorem}
\begin{proof}
We explicitly construct the corresponding translation surface for $G\in\mathcal{P} \sqcup \mathcal{Q}$.
Now, there are two cases here.

Case 1.
Assume $G\in \mathcal{P}$. So, $G=G_{n,m}$ for some $n\ge 2, m\in \N$. Then, the vertices of $G$ are labeled as $v_0, v_1, v_2, \dots, v_{n-1}$ such that there are $2m$ parallel edges joining the pair of vertices ($v_i$, $v_{i+1}$) mod $n$, $i= 0,1,2, \dots, n-1$. Consider the permutation $\sigma$ = ($0$ $1$ $2$ \dots $n-1$) in $S_n$. Now, we form our desired translation surface as follows:

First, we take $n$ numbers of regular $4m$ gons denoted by $P_0$, $P_1$, $P_2$, \dots, $P_{n-1}$. Then we label the sides of $P_i$ as $E_1(i), E_2(i), \dots, E_{2m}(i), \bar{E_1}(i), \bar{E_2}(i), \dots \bar{E}_{2m}(i)$ such that $E_k(i)$ is parallel to $\bar{E_k}(i)$, $1\leq k\leq 2m$ for each $i\in \{0,1,2 \dots, n-1\}$. Now for a fixed $i\in\{0, 1, 2, \dots, n-1\}$ we glue the sides $E_k(i)$ and $\bar{E_k}(\sigma(i))$, $1\leq k\leq 2m$  by translation. 
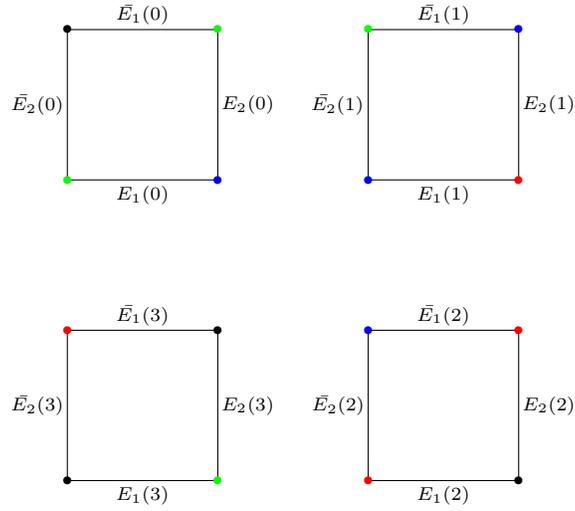
\begin{figure}[htbp]
\begin{center}
\begin{tikzpicture}[scale=1]

\draw (-3,-4)--(-1,-4)--(-1,-2)--(-3,-2)--(-3,-4);
\draw (3,-4)--(1,-4)--(1,-2)--(3,-2)--(3,-4);

\draw (-3,-8)--(-1,-8)--(-1,-6)--(-3,-6)--(-3,-8);
\draw (3,-8)--(1,-8)--(1,-6)--(3,-6)--(3,-8);

\draw (-2,-4.2) node{\tiny$E_1(0)$};
\draw (-0.6, -3 ) node{\tiny$E_2(0)$};
\draw (-2,-1.8) node{\tiny$\bar{E_1}(0)$};
\draw (-3.4,-3) node{\tiny$\bar{E_2}(0)$};

\draw (2,-1.8) node{\tiny$\bar{E_1}(1)$};
\draw (0.59,-3) node{\tiny$\bar{E_2}(1)$};
\draw (-2,-8.2) node{\tiny$E_1(3)$};
\draw (-0.6,-7) node{\tiny$E_2(3)$};

\draw (2,-4.2) node{\tiny$E_1(1)$};
\draw (3.4,-3) node{\tiny$E_2(1)$};
\draw (2,-5.8) node{\tiny$\bar{E_1}(2)$};
\draw (0.6,-7) node{\tiny$\bar{E_2}(2)$};
\draw (2,-8.2) node{\tiny$E_1(2)$};
\draw (3.4,-7) node{\tiny$E_2(2)$};
\draw (-2,-5.8) node{\tiny$\bar{E_1}(3)$};
\draw (-3.4,-7) node{\tiny$\bar{E_2}(3)$};

\draw[green] (-3,-4) node{\tiny$\bullet$};
\draw[green] (-1,-2) node{\tiny$\bullet$};
\draw[green] (1,-2) node{\tiny$\bullet$};
\draw[green] (-1,-8) node{\tiny$\bullet$};

\draw[blue] (-1,-4) node{\tiny$\bullet$};
\draw[blue] (1,-4) node{\tiny$\bullet$};
\draw[blue] (3,-2) node{\tiny$\bullet$};
\draw[blue] (1,-6) node{\tiny$\bullet$};

\draw (-3,-2) node{\tiny$\bullet$};
\draw (-3,-8) node{\tiny$\bullet$};
\draw (-1,-6) node{\tiny$\bullet$};
\draw (3,-8) node{\tiny$\bullet$};

\draw[red] (-3,-6) node{\tiny$\bullet$};
\draw[red] (1,-8) node{\tiny$\bullet$};
\draw[red] (3,-6) node{\tiny$\bullet$};
\draw[red] (3,-4) node{\tiny$\bullet$};

\end{tikzpicture}
\caption{Translation surface for $G_{4,1}$.}
\label{figure16}
\end{center}
\end{figure}

In this way we get a translation surface $S$ with $n$ singular points (or marked points) and each sides of $P_i$'s becomes a systolic connection of $S$. So, it is clear that $\Gamma_S \cong G$. Also by the construction $S\setminus G$ is nothing but union of the $n$ number of regular $4m$ gons i.e. topological disks and hence $G$ admits cellular-systolic embedding on $S$.
 For example, Figure~\ref{figure16} illustrates the construction of a translation surface for $G_{4,1}$.

Case 2.
Assume $G\in \mathcal{Q}$.
This implies that $G=G_{n,m}^\prime$ for some $n\ge2, m\in \N$.
 So, label the vertices of $G$ as $v_0$, $v_1$, $v_2$, \dots, $v_{n-1}$ such that,
 (1) each vertex $v_i$, $i= 0,1,\dots, n-1$ has $m$ loops and  
 (2) there are $m$ parallel edges joining the pair of vertices ($v_i$, $v_{i+1}$) mod $n$, $0\le i\le n-1$.
 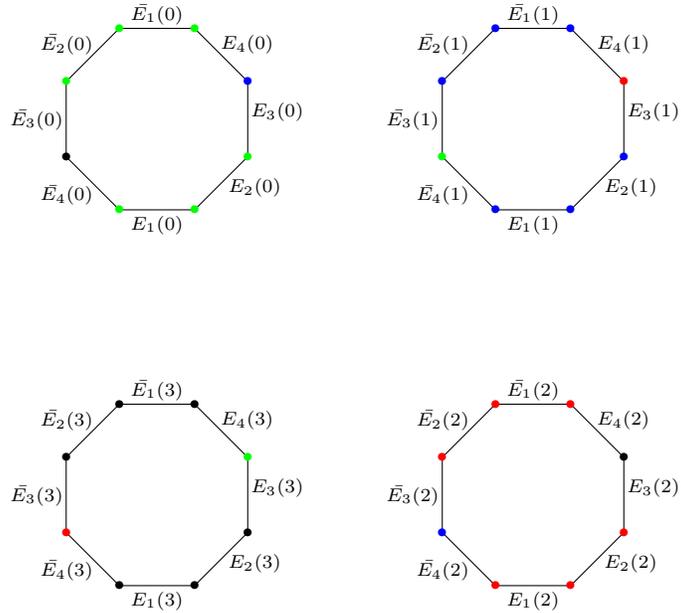
\begin{figure}[htbp]
\begin{center}
\begin{tikzpicture}[scale=1]
\draw (0,-5)--(1,-5);
\draw (1,-5)--(1.707,-4.293);
\draw (1.707,-4.293)--(1.707,-3.293);
\draw  (1.707,-3.293)--(1,-2.586);
\draw (0, -2.586)--(1,-2.586);
\draw (0,-2.586)--(-0.707,-3.293);
\draw (-0.707,-3.293)--(-0.707,-4.293);
\draw (-0.707,-4.293)--(0,-5);

\draw (0.5,-5.2) node{\tiny$E_1(0)$};
\draw (1.8,-4.7) node{\tiny$E_2(0)$};
\draw (2.1,-3.7) node{\tiny$E_3(0)$};
\draw (1.7,-2.8) node{\tiny$E_4(0)$};
\draw (0.5,-2.4) node{\tiny$\bar{E_1}(0)$};
\draw (-0.7,-2.8) node{\tiny$\bar{E_2}(0)$};
\draw (-1.1,-3.8) node{\tiny$\bar{E_3}(0)$};
\draw (-0.7,-4.8) node{\tiny$\bar{E_4}(0)$};

\draw[green] (0,-5) node{\tiny$\bullet$};
\draw[green] (1,-5) node{\tiny$\bullet$};
\draw[green] (1.707,-4.293) node{\tiny$\bullet$};
\draw[blue] (1.707,-3.293) node{\tiny$\bullet$};
\draw[green] (1,-2.586) node{\tiny$\bullet$};
\draw[green] (0,-2.586) node{\tiny$\bullet$};
\draw[green] (-0.707,-3.293) node{\tiny$\bullet$};
\draw (-0.707,-4.293) node{\tiny$\bullet$};

\draw (5,-5)--(6,-5);
\draw (6,-5)--(6.707,-4.293);
\draw (6.707,-4.293)--(6.707,-3.293);
\draw  (6.707,-3.293)--(6,-2.586);
\draw (5, -2.586)--(6,-2.586);
\draw (5,-2.586)--(4.293,-3.293);
\draw (4.293,-3.293)--(4.293,-4.293);
\draw (4.293,-4.293)--(5,-5);

\draw (5.5,-5.2) node{\tiny$E_1(1)$};
\draw (6.8,-4.7) node{\tiny$E_2(1)$};
\draw (7.1,-3.7) node{\tiny$E_3(1)$};
\draw (6.7,-2.8) node{\tiny$E_4(1)$};
\draw (5.5,-2.4) node{\tiny$\bar{E_1}(1)$};
\draw (4.3,-2.8) node{\tiny$\bar{E_2}(1)$};
\draw (3.9,-3.8) node{\tiny$\bar{E_3}(1)$};
\draw (4.3,-4.8) node{\tiny$\bar{E_4}(1)$};

\draw[blue] (5,-5) node{\tiny$\bullet$};
\draw[blue] (6,-5) node{\tiny$\bullet$};
\draw[blue] (6.707,-4.293) node{\tiny$\bullet$};
\draw[red] (6.707,-3.293) node{\tiny$\bullet$};
\draw[blue] (6,-2.586) node{\tiny$\bullet$};
\draw[blue] (5,-2.586) node{\tiny$\bullet$};
\draw[blue] (4.293,-3.293) node{\tiny$\bullet$};
\draw[green] (4.293,-4.293) node{\tiny$\bullet$};

\draw (5,-10)--(6,-10);
\draw (6,-10)--(6.707,-9.293);
\draw (6.707,-9.293)--(6.707,-8.293);
\draw  (6.707,-8.293)--(6,-7.586);
\draw (5, -7.586)--(6,-7.586);
\draw (5,-7.586)--(4.293,-8.293);
\draw (4.293,-8.293)--(4.293,-9.293);
\draw (4.293,-9.293)--(5,-10);

\draw (5.5,-10.2) node{\tiny$E_1(2)$};
\draw (6.8,-9.7) node{\tiny$E_2(2)$};
\draw (7.1,-8.7) node{\tiny$E_3(2)$};
\draw (6.7,-7.8) node{\tiny$E_4(2)$};
\draw (5.5,-7.4) node{\tiny$\bar{E_1}(2)$};
\draw (4.3,-7.8) node{\tiny$\bar{E_2}(2)$};
\draw (3.9,-8.8) node{\tiny$\bar{E_3}(2)$};
\draw (4.3,-9.8) node{\tiny$\bar{E_4}(2)$};

\draw[red] (5,-10) node{\tiny$\bullet$};
\draw[red] (6,-10) node{\tiny$\bullet$};
\draw[red] (6.707,-9.293) node{\tiny$\bullet$};
\draw (6.707,-8.293) node{\tiny$\bullet$};
\draw[red] (6,-7.586) node{\tiny$\bullet$};
\draw[red] (5,-7.586) node{\tiny$\bullet$};
\draw[red] (4.293,-8.293) node{\tiny$\bullet$};
\draw[blue] (4.293,-9.293) node{\tiny$\bullet$};

\draw (0,-10)--(1,-10);
\draw (1,-10)--(1.707,-9.293);
\draw (1.707,-9.293)--(1.707,-8.293);
\draw  (1.707,-8.293)--(1,-7.586);
\draw (0, -7.586)--(1,-7.586);
\draw (0,-7.586)--(-0.707,-8.293);
\draw (-0.707,-8.293)--(-0.707,-9.293);
\draw (-0.707,-9.293)--(0,-10);

\draw (0.5,-10.2) node{\tiny$E_1(3)$};
\draw (1.8,-9.7) node{\tiny$E_2(3)$};
\draw (2.1,-8.7) node{\tiny$E_3(3)$};
\draw (1.7,-7.8) node{\tiny$E_4(3)$};
\draw (0.5,-7.4) node{\tiny$\bar{E_1}(3)$};
\draw (-0.7,-7.8) node{\tiny$\bar{E_2}(3)$};
\draw (-1.1,-8.8) node{\tiny$\bar{E_3}(3)$};
\draw (-0.7,-9.8) node{\tiny$\bar{E_4}(3)$};

\draw (0,-10) node{\tiny$\bullet$};
\draw (1,-10) node{\tiny$\bullet$};
\draw (1.707,-9.293) node{\tiny$\bullet$};
\draw[green] (1.707,-8.293) node{\tiny$\bullet$};
\draw (1,-7.586) node{\tiny$\bullet$};
\draw (0,-7.586) node{\tiny$\bullet$};
\draw (-0.707,-8.293) node{\tiny$\bullet$};
\draw[red] (-0.707,-9.293) node{\tiny$\bullet$};

\end{tikzpicture}
\caption{Translation surface for $G_{4,2}^\prime$.}
\label{figure17}
\end{center}
\end{figure}
Also, in this case we  consider the permutation $\sigma$ = ($0$ $1$ $2$ \dots $n-1$) from $S_n$. Similarly to case $1$, we take $n$ numbers of regular $4m$ gons denoted by $P_0$, $P_1$, $P_2$, \dots, $P_{n-1}$. Label the sides of $P_i$ as $E_1(i), E_2(i), \dots, E_{2m}(i), \bar{E_1}(i), \bar{E_2}(i), 
\dots \bar{E}_{2m}(i)$ such that $E_k(i)$ is parallel to $\bar{E_k}(i)$, $1\leq k\leq m$ for each $i\in \{0,1,2 \dots, n-1\}$. Here, we form the translation surface as follows:

For a fixed $i\in \{0,1,2 \dots, n-1\}$ we identify the sides $E_k(i)$ and $\bar{E}_k(i)$  for $1\le k\le m$ and the sides $E_k(i)$ and $\bar{E}_k(\sigma(i))$ for $m+1\le k\le 2m$ via translation. So, we get a translation surface $S$ from this with $n$ number of singular points (or marked points) and each side of $P_i$'s is systolic connection of S. Therefore, $G$ admits cellular-systolic embedding on $S$. In Figure \ref{figure17}, a translation surface is explicitly constructed for $G_{4,2}^\prime$. 

The genus of the translation surface for both the cases is $\left[n(m-1)+1\right]$.

Therefore, from both the cases we conclude that any $G\in \mathcal{P} \sqcup \mathcal{Q}$ admits cellular-systolic embedding on a translation surface.

\end{proof}

\bibliographystyle{alpha}
\bibliography{bibliography}
\end{document}